%% file: main.tex
\documentclass[graybox]{svmult}

\usepackage{type1cm}        
\usepackage{trimclip}
\usepackage{makeidx}         
\usepackage{graphicx}        
\usepackage{multicol}        
\usepackage[bottom]{footmisc}
\usepackage{hyperref}

\usepackage{newtxtext}       %
\usepackage[varvw]{newtxmath}       

\makeindex             

\newcommand{\modcat}{{\ref{ch:model:eq:ISO3}}}

\usepackage{tikz}
\usepackage{tikz-3dplot}
\definecolor{babyblue}{rgb}{0.54, 0.81, 0.94}
\usepackage{mathtools}
\usetikzlibrary{calc}

\definecolor{emerald}{rgb}{0.31, 0.78, 0.47}

\definecolor{celestialblue}{rgb}{0.29, 0.59, 0.82}

\begin{document}

\title*{$p$-Wasserstein distances on networks and 3D to 1D convergence}
\author{Martin Burger\orcidID{0000-0003-2619-2912} and\\ Ariane Fazeny\orcidID{0009-0002-0008-3646} and\\ Gilles Mordant\orcidID{0000-0003-1821-5168} and\\ Jan-Frederik Pietschmann\orcidID{0000-0003-0383-8696}}
\institute{Martin Burger \at Helmholtz Imaging, Deutsches Elektronen-Synchrotron DESY, Notkestr. 85, 22607 Hamburg, Germany. \email{martin.burger@desy.de} \newline Fachbereich Mathematik, Universit\"at Hamburg, Bundesstrasse
55, 20146 Hamburg, Germany. \email{martin.burger@uni-hamburg.de}
\and Ariane Fazeny \at Helmholtz Imaging, Deutsches Elektronen-Synchrotron DESY, Notkestr. 85, 22607 Hamburg, Germany. \email{ariane.fazeny@desy.de}
\and
Gilles Mordant \at Applied and Computational Mathematics, Yale University. \email{gilles.mordant@yale.edu}
\and Jan-Frederik Pietschmann \at \small
Universit\"{a}t Augsburg, Institut f\"ur Mathematik and
Centre for Advanced Analytics and Predictive Sciences (CAAPS), Universit\"{a}tsstra\ss e 12a, 86159 Augsburg, Germany.
\email{jan-f.pietschmann@uni-a.de}}
%
%
\maketitle

\abstract*{Each chapter should be preceded by an abstract (no more than 200 words) that summarizes the content. The abstract will appear \textit{online} at \url{www.SpringerLink.com} and be available with unrestricted access. This allows unregistered users to read the abstract as a teaser for the complete chapter.
Please use the 'starred' version of the \texttt{abstract} command for typesetting the text of the online abstracts (cf. source file of this chapter template \texttt{abstract}) and include them with the source files of your manuscript. Use the plain \texttt{abstract} command if the abstract is also to appear in the printed version of the book.}

\abstract{
We study transport distances on metric graphs representing gas networks. Starting from the dynamic formulation of the Wasserstein distance, we review extensions to networks, with and without the possibility of storing mass on the vertices.
Next, we examine the asymptotic behavior of the static Wasserstein distance on a three-dimensional network domain that converges to a metric graph. We show convergence of the distance with a proof that is based on the characterization of optimal transport plans as $c$-cyclically monotone sets.
We conclude by illustrating our finding with several numerical examples.}

\input{commands}
\input{0_Introduction}

\include{1_Notation_Preliminaries}
\include{2_Wasserstein_Results}

\input{3_3D}
\input{4_Main_Results}


\begin{acknowledgement}
JFP thanks the DFG for support via the Research Unit FOR 5387 POPULAR, Project No. 461909888.
AF and MB thank the DFG for support via the SFB TRR 154 and acknowledge support from DESY (Hamburg, Germany), a member of the Helmholtz Association HGF.
\end{acknowledgement}
\ethics{Competing Interests}{none.}

\eject

\input{references}

\end{document}

%% file: commands.tex
\def\calA{{\mathcal A}} \def\calB{{\mathcal B}} \def\calC{{\mathcal C}}
\def\calD{{\mathcal D}} \def\calE{{\mathcal E}} \def\calF{{\mathcal F}}
\def\calG{{\mathcal G}} \def\calH{{\mathcal H}} \def\calI{{\mathcal I}}
\def\calJ{{\mathcal J}} \def\calK{{\mathcal K}} \def\calL{{\mathcal L}}
\def\calM{{\mathcal M}} \def\calN{{\mathcal N}} \def\calO{{\mathcal O}}
\def\calP{{\mathcal P}} \def\calQ{{\mathcal Q}} \def\calR{{\mathcal R}}
\def\calS{{\mathcal S}} \def\calT{{\mathcal T}} \def\calU{{\mathcal U}}
\def\calV{{\mathcal V}} \def\calW{{\mathcal W}} \def\calX{{\mathcal X}}
\def\calY{{\mathcal Y}} \def\calZ{{\mathcal Z}}
\def\rma{{\mathrm a}} \def\rmb{{\mathrm b}} \def\rmc{{\mathrm c}}
\def\rmd{{\mathrm d}} \def\rme{{\mathrm e}} \def\rmf{{\mathrm f}}
\def\rmg{{\mathrm g}} \def\rmh{{\mathrm h}} \def\rmi{{\mathrm i}}
\def\rmj{{\mathrm j}} \def\rmk{{\mathrm k}} \def\rml{{\mathrm l}}
\def\rmm{{\mathrm m}} \def\rmn{{\mathrm n}} \def\rmo{{\mathrm o}}
\def\rmp{{\mathrm p}} \def\rmq{{\mathrm q}} \def\rmr{{\mathrm r}}
\def\rms{{\mathrm s}} \def\rmt{{\mathrm t}} \def\rmu{{\mathrm u}}
\def\rmv{{\mathrm v}} \def\rmw{{\mathrm w}} \def\rmx{{\mathrm x}}
\def\rmy{{\mathrm y}} \def\rmz{{\mathrm z}}
\def\rmA{{\mathrm A}} \def\rmB{{\mathrm B}} \def\rmC{{\mathrm C}}
\def\rmD{{\mathrm D}} \def\rmE{{\mathrm E}} \def\rmF{{\mathrm F}}
\def\rmG{{\mathrm G}} \def\rmH{{\mathrm H}} \def\rmI{{\mathrm I}}
\def\rmJ{{\mathrm J}} \def\rmK{{\mathrm K}} \def\rmL{{\mathrm L}}
\def\rmM{{\mathrm M}} \def\rmN{{\mathrm N}} \def\rmO{{\mathrm O}}
\def\rmP{{\mathrm P}} \def\rmQ{{\mathrm Q}} \def\rmR{{\mathrm R}}
\def\rmS{{\mathrm S}} \def\rmT{{\mathrm T}} \def\rmU{{\mathrm U}}
\def\rmV{{\mathrm V}} \def\fw{{\mathrm W}} \def\rmX{{\mathrm X}}
\def\rmY{{\mathrm Y}} \def\rmZ{{\mathrm Z}}

\newcommand{\bdiv}{\mathop{\mathsf{div}}\nolimits}
\newcommand{\gdiv}{\mathop{\overline{\mathrm{div}}}\nolimits}
\newcommand{\pderiv}[3][]{\frac{\mathop{}\!\mathrm{d}^{#1} #2}{\mathop{}\!\mathrm{d} #3^{#1}}}

\def\sfa{{\mathsf a}} \def\sfb{{\mathsf b}} \def\sfc{{\mathsf c}}
\def\sfd{{\mathsf d}} \def\sfe{{\mathsf e}} \def\sff{{\mathsf f}}
\def\sfg{{\mathsf g}} \def\sfh{{\mathsf h}} \def\sfi{{\mathsf i}}
\def\sfj{{\mathsf j}} \def\sfk{{\mathsf k}} \def\sfl{{\mathsf l}}
\def\sfm{{\mathsf m}} \def\sfn{{\mathsf n}} \def\sfo{{\mathsf o}}
\def\sfp{{\mathsf p}} \def\sfq{{\mathsf q}} \def\sfr{{\mathsf r}}
\def\sfs{{\mathsf s}} \def\sft{{\mathsf t}} \def\sfu{{\mathsf u}}
\def\sfv{{\mathsf v}} \def\sfw{{\mathsf w}} \def\sfx{{\mathsf x}}
\def\sfy{{\mathsf y}} \def\sfz{{\mathsf z}}

\def\sfA{{\mathsf A}} \def\sfB{{\mathsf B}} \def\sfC{{\mathsf C}}
\def\sfD{{\mathsf D}} \def\sfE{{\mathsf E}} \def\sfF{{\mathsf F}}
\def\sfG{{\mathsf G}} \def\sfH{{\mathsf H}} \def\sfI{{\mathsf I}}
\def\sfJ{{\mathsf J}} \def\sfK{{\mathsf K}} \def\sfL{{\mathsf L}}
\def\sfM{{\mathsf M}} \def\sfN{{\mathsf N}} \def\sfO{{\mathsf O}}
\def\sfP{{\mathsf P}} \def\sfQ{{\mathsf Q}} \def\sfR{{\mathsf R}}
\def\sfS{{\mathsf S}} \def\sfT{{\mathsf T}} \def\sfU{{\mathsf U}}
\def\sfV{{\mathsf V}} \def\sfW{{\mathsf W}} \def\sfX{{\mathsf X}}
\def\sfY{{\mathsf Y}} \def\sfZ{{\mathsf Z}} 


\makeatletter
\NewDocumentCommand{\rmjmath}{}{\mathbin{\mathpalette\eplus@\relax\mspace{1mu}}}
\newcommand{\eplus@}[2]{\clipbox{-.5 -.5 0 {.35\height}}{$\m@th#1\rmj$}}
\makeatother

\DeclarePairedDelimiter{\bra}{(}{)}
\DeclarePairedDelimiter{\pra}{[}{]}
\DeclarePairedDelimiter{\set}{\{}{\}}
\DeclarePairedDelimiter{\skp}{\langle}{\rangle}

\def\scra{{\mathscr  a}} \def\scrb{{\mathscr  b}} \def\scrc{{\mathscr  c}}
\def\scrd{{\mathscr  d}} \def\scre{{\mathscr  e}} \def\scrf{{\mathscr  f}}
\def\scrg{{\mathscr  g}} \def\scrh{{\mathscr  h}} \def\scri{{\mathscr  i}}
\def\scrj{{\mathscr  j}} \def\scrk{{\mathscr  k}} \def\scrl{{\mathscr  l}}
\def\scrm{{\mathscr  m}} \def\scrn{{\mathscr  n}} \def\scro{{\mathscr  o}}
\def\scrp{{\mathscr  p}} \def\scrq{{\mathscr  q}} \def\scrr{{\mathscr  r}}
\def\scrs{{\mathscr  s}} \def\scrt{{\mathscr  t}} \def\scru{{\mathscr  u}}
\def\scrv{{\mathscr  v}} \def\scrw{{\mathscr  w}} \def\scrx{{\mathscr  x}}
\def\scry{{\mathscr  y}} \def\scrz{{\mathscr  z}} 

\newcommand{\N}{{\mathbb{N}}}
\newcommand{\diffedge}{d}
\newcommand{\AC}{\mathrm{AC}}
\newcommand{\Mgraph}{\sfM}
\newcommand{\Medges}{\sfL}
\newcommand{\nodes}{\sfV}
\newcommand{\edges}{\sfE}
\newcommand{\normal}{\sfn}
\newcommand{\R}{\mathbb{R}}
\newcommand{\EDP}{{\mathsf{EDP}}}
\newcommand{\CE}{{\mathsf{CE}}}

\newcommand{\pCE}{{\overline{\mathsf{CE}}_{\hat\nodes}}}

\newcommand{\Prb}[1]{{\mathcal{P}\!\left({#1}\right)}}
\newcommand{\Jprb}[1]{{\Pi\left({#1}\right)}}
\newcommand{\Mea}[1]{{\mathcal{M}\left({#1}\right)}}

\newcommand{\Om}{{\Omega}}
\newcommand{\Ome}{{\Omega_\varepsilon}}
\newcommand{\Omo}{{\Omega_0}}
\newcommand{\Omov}{{\overline{\Omega}}}
\newcommand{\Be}{{B_\varepsilon}}
\newcommand{\Ne}{{\mathcal{N}_\epsilon}}
\newcommand{\No}{{\mathcal{N}_0}}
\newcommand{\Nov}{{\overline{\mathcal{N}}}}

\newcommand{\rnd}[1]{{\left({#1}\right)}}
\newcommand{\sqr}[1]{{\left[{#1}\right]}}
\newcommand{\crl}[1]{{\left\{{#1}\right\}}}
\newcommand{\crlm}[2]{{\left\{{#1} ~ \middle| ~ {#2}\right\}}}
\newcommand{\res}[2]{{\left.{#1}\right\rvert_{#2}}}

\newcommand{\norm}[1]{{\left\lvert\left\lvert{#1}\right\rvert\right\rvert}}
\newcommand{\abs}[1]{{\left\lvert{#1}\right\rvert}}
\newcommand{\dd}{{\, \mathrm d}}

\newcommand{\eps}{{\varepsilon}}
\newcommand{\G}{{\mathcal{G}}}
\newcommand{\Eo}{{\overline{\mathcal{E}}}}
\newcommand{\vel}{{\mathrm{v}}}

\newcommand{\mue}{{\mu_\varepsilon}}
\newcommand{\muo}{{\mu_0}}
\newcommand{\muov}{{\overline{\mu}}}
\newcommand{\nue}{{\nu_\varepsilon}}
\newcommand{\nuo}{{\nu_0}}
\newcommand{\nuov}{{\overline{\nu}}}
\newcommand{\pie}{{\pi_\varepsilon}}
\newcommand{\pio}{{\pi_0}}
\newcommand{\piov}{{\overline{\pi}}}
\newcommand{\ce}{{c_{\!\varepsilon}}}
\newcommand{\co}{{c_0}}
\newcommand{\cov}{{\overline{c}}}

%% file: 0_Introduction.tex
\section{Introduction} \label{sec: Intro}

The theory of optimal transport, introduced by G. Monge in \cite{Monge1781} and reformulated by L.V. Kantorovich in \cite{Kantorovich1942}, is nowadays an ubiquitous tool in many areas of applied mathematics, statistics, and data science \cite{PeyreCuturi2019, Santambrogio2015, Cuturi2013, Arjovsky2017, Panaretos2019,Kolouri2017, Frogner2015, Hallin2020MultivariateGT, lavenant2021towards, zhang2023manifold}. The theory arose from the ``logistic'' problem of moving mass from a source configuration to a target configuration at overall minimal cost. Here, the ground cost encodes the difficulty of transporting one unit of mass from a point $x$ to a point $y$. As such, the overall optimal cost thus defines a metric (or more general distance functional) between probability distributions that leverages the ground metric, which makes it a versatile tool. \\

Additionally to the optimal transport plan, establishing how to reorganize the mass from the source to the target, it is sometimes possible to add a time component to the problem and track how the reorganization should evolve as time passes. The problem without the time component is often called the \emph{static formulation}, while the other version of the problem is called the \emph{dynamic formulation}, provided it exists. \\

This dynamic formulation further connects certain evolution equations to gradient flows in the space of probability measures endowed with the Wasserstein metric. This view,  motivated by physical principles \cite{JordanKinderlehrerOtto1997,JordanKinderlehrerOtto1998, Otto2001, MielkePeletierRenger2014, Mielke2023}, has had a tremendous impact in the PDE community. In the Euclidean case, a gradient flow is given by a smooth curve $x: \R_{\geq 0} \to \Om, ~ x\rnd{t} = x_t$ for some space domain $\Om$, such that $\mathrm d x_t / \mathrm d t = - \nabla E\rnd{x_t}$ for $t \geq 0$, where $E: \Om \to \R$ is some (typically convex) energy function. Then the gradient flow $x_t$ minimizes $E$ following the steepest descent. If we now want to evaluate an energy function at not only a point in space $x_t \in \Om$, but rather for a probability distribution $\rho_t \in \Prb{\Om}$ on the space domain $\Om$, then in order to have a well-defined gradient of $E$ with respect to probability distributions we need a way of quantifying the distance between two probability measures. \\

One popular metric for quantifying the distance (in the sense of optimal transport cost) between two probability measures is the  the \textbf{Wasserstein distance}. Given a domain $\Omega \subset \R^d$, $d\in\N$, and two probability measures $\mu,\,\nu \in \Prb{\Om}$, the minimal total transport cost of going from $\mu$ to $\nu$ is given by
\begin{equation} \label{WDS} \tag{WDS}
    W_p \rnd{\mu, \nu} \coloneqq \min_{\pi \in \Jprb{\mu, \nu}} \crl{\int_{\Om \times \Om} \norm{x - y}^p \dd \pi\rnd{x, y}}^{1/p},
\end{equation}
where $\Jprb{\mu, \nu}$ denotes the set of all joint probability distributions $\pi$ on $\Om \times \Om$, with the respective marginals $\mu$ and $\nu$, meaning that $\mu\rnd{A} = \pi\rnd{A, \Omega}$ and $\nu\rnd{B} =  \pi\rnd{\Omega, B}$ for all Borel measurable sets $A,B \subset \Omega$. In the definition above, the ground cost, generally denoted by $c$, is the $p$-th power of the Euclidean distance. Further, this formulation of the Wasserstein distance solely involves an infimum over so-called couplings $\pi$ on $\Om \times \Om$ without any time component, thus it corresponds to the static formulation. \\

{In the case of a non-convex domain $\Om$, the formulation \eqref{WDS} generally also allows mass transport outside of $\Om$, as the costs only consider the Euclidean distance. In order to restrict the feasible transport to the space domain $\Om$, one can alternatively also define the static formulation of the Wasserstein distance as
\begin{equation}  
    W_p \rnd{\mu, \nu} \coloneqq \min_{\pi \in \Jprb{\mu, \nu}} \crl{\int_{\Om \times \Om} d_{\Omega,p}\rnd{x,y}^p \dd \pi\rnd{x, y}}^{1/p},
\end{equation}
where $d_{\Omega,p}$ is the shortest path distance on $\Omega$ given by
\begin{equation}
    d_{\Omega,p}\rnd{x,y} = \inf_{\substack{\xi \in C^1\rnd{\sqr{0,1}}, \\ \xi\rnd{0}=x, \, \xi\rnd{1}=y}} \crl{\int_0^1 \norm{\dot{\xi}\rnd{t}}_p^p \dd t}^{1/p}.
\end{equation}
Here, the infimum is taken over all (continuous) curves $\xi$ in $\Omega$ which are connecting $x\in\Om$ to $y\in\Om$.} \\

The optimization problem (\ref{WDS}) is linear in the coupling $\pi$, which also needs to fulfill the marginal constraints. Hence, it is natural to express the constraints by relying on Lagrange multipliers. Doing so, one obtains an ``inf-sup'' problem where it turns out that the \textit{inf} and the \textit{sup} can be exchanged. This crucial fact gives rise to a dual formulation of the problem, which we recall later in Theorem~\ref{thm: KantDual}. \\

A further important aspect is the notion of $c$-cyclical monotonicity of the set on which an optimal coupling $\pi^\star$ for the problem (\ref{WDS}) is concentrated, see Definition~\ref{def: Cycl}. The $c$-cyclical monotonicity expresses the fact that the optimality of $\pi^\star$ would be contradicted  by the possibility to re-route the mass of $\pi^\star$ to produce a coupling achieving a smaller overall cost. Such considerations thus induce strong geometric constraints on the support of $\pi^\star$, which will be leveraged in some of our following results. \\

In the seminal contribution of J.D. Benamou and Y. Brenier in 2000, \cite{benamou2000computational}, the minimization problem \eqref{WDS} was shown to have an equivalent dynamic formulation. The main idea is to consider curves of probability measures $\rnd{\rho_t}_{0\le t \le 1} \subset \Prb{\Om}$ over an artificial time interval $[0,1]$, which connect the two measures $\mu$ and $\nu$, such that $\rho_0 = \mu$ and $\rho_1 = \nu$. As mass is preserved along this evolution from $\mu$ to $\nu$, the curves $(\rho_t)_{0\le t \le 1}$ are characterized by solutions (in the sense of distributions) to the continuity equation, i.e.
\begin{equation}\label{eq:cont_intro}
    \begin{aligned}
        \partial_t \rho_t + \nabla \cdot (\rho_tv_t) = 0,\\
        \rho_0 = \mu,\quad \rho_1 = \nu,
    \end{aligned}
\end{equation}
for some (measure-valued) velocity field $v_t$. We denote the set of all such curves $\rnd{\rho_t}_{0\le t \le 1} \subset \Prb{\Om}$ fulfilling \eqref{eq:cont_intro} by $\CE\rnd{\mu,\nu}$. Then, selecting the curve with minimal kinetic energy from $\CE\rnd{\mu,\nu}$, thus corresponding to an optimal transport from $\mu$ to $\nu$, one recovers the $2$-Wassertein distance, i.e.,
\begin{align}\label{eq:dynamic_action} \tag{2-WDD}
    W_2\rnd{\mu,\nu} = \inf_{\rnd{\rho_t,v_t}\in\CE\rnd{\mu,\nu}} \crl{\int_0^1\int_\Omega \abs{v_t}^2 \dd \rho_t}^{1/2}.
\end{align}

The connection of the dynamic formulation of the Wasserstein distance to the problem of gas transport within a complicated network is done by extending \eqref{eq:dynamic_action} to metric graphs as the domain $\Om$, modeling a gas network. Roughly speaking, metric graphs are combinatorial graphs in which a one-dimensional interval is associated to each edge (see subsection \ref{sec:met_graph} for a rigorous definition). Then, the continuity equation \eqref{eq:cont_intro} can be used on each of the edges to ensure mass conservation in the individual pipes. However, to enforce global mass conservation on the entire network, additional coupling conditions at the vertices, corresponding to pipe junctions, have to be introduced. \\

So far, two approaches have been considered in the literature: \newline
In \cite{erbar2021gradient}, homogeneous Kirchhoff conditions were imposed. In this setting, the sum of mass entering the node is equal to the mass leaving it, which means that no mass is stored on the node and no additional gas is entering the network through the node. This also allows for a static formulation on the metric graph, analogous to \eqref{WDS}, by choosing a metric graph as the domain $\Om$ and replacing the Euclidean norm distance by the distance (\ref{def:graph_dist}) on the metric graph, defined in subsection \ref{sec:met_graph}. \newline
The second approach, introduced in \cite{burger2023dynamic}, allows for storage of mass on a node, utilizing a more general version of Kirchhoff's law. Here, additional continuity equations on the vertices are imposed. It was then a striking observation, made in \cite{Fazeny_2025}, and rigorously proven in \cite{weigand2025pdiffusion}, that the ISO3 model for gas flow, given as 
\begin{align}
    \label{eqn:ISO3}
    \begin{split}
        \partial_t \rho^\sfe + \partial_x\rnd{\rho^\sfe v^\sfe} & = 0 \\
        \partial_x p^\sfe\rnd{\rho^\sfe} & = -\frac{\lambda^\sfe}{2 D^\sfe} \rho^\sfe v^\sfe \abs{v^\sfe} - g \rho^\sfe \sin{\left(\omega^\sfe\right)}
    \end{split}
    \tag{ISO3}
\end{align}
can be interpreted as a 3-Wasserstein gradient flow for a certain energy functional $E$. Here, the superscript-$\sfe$-notation corresponds to the restriction of the respective functions defined on the whole gas network to a single pipe $\sfe$. The derivation of the ISO3 model is described in detail in \modcat, alongside an entire hierarchy of models for the gas flow within one pipe, all derived from the Euler equations system for compressible, inviscid fluids. Note that, here the pipe inclination angle $\omega_e$ relative to the ground is used instead of the pipe hight function $b_e$. Both approaches for the coupling conditions at vertices, \cite{erbar2021gradient} and \cite{burger2023dynamic}, are based on a dynamic formulation of the Wasserstein distance on a metric graph, where the second approach is closely related to the study of transport on coupled domains of different dimension (the 1-dimensional edges and the 0-dimensional vertices), such as \cite{monsaingeon2021new, CarioniKrautzPietschmann2025_PreferentialPaths, CarioniKrautzPietschmann2025_StarShapedGraphs}. 

While the setting of metric graphs is a convenient way to model gas networks, it has the inherent limitation that pipes are modeled as one-dimensional entities, only, utilizing an averaging argument over the pipe cross sections. This gives rise to the question whether this approach can be justified by a limiting procedure, starting from a fully three-dimensional model and, in turn, letting the diameter of the pipes go to zero. We study this question in Section~\ref{sec:3d_to_1d} and are indeed able to prove the convergence of the Wasserstein metric in its static formulation (corresponding to the case without mass storage on vertices). This serves as a starting point for further investigations of the convergence of the (equivalent) dynamic formulation and, subsequently, also the convergence of Wasserstein gradient flows.

%% file: 2_Wasserstein_Results.tex
\section{Dynamic transport distances on metric graphs} \label{sec: Wasserstein}

Our next aim is to give a brief review of dynamic transport distances on metric graphs and their connection to the ISO3 model, following the results of \cite{erbar2021gradient,Heinze2024,Fazeny_2025, weigand2025pdiffusion}. 

\subsection{Metric graphs} \label{sec:met_graph}

Using the notation of \cite{Heinze2024}, we start with an undirected 
connected graph, which is a (finite) node set $\nodes$ with edge set $\edges \subset \nodes\times \nodes$, such that if $\sfv\sfw \in \edges$ then $\sfw\sfv\not\in\edges$. We use $\edges(\sfv)$ to denote the set of all edges adjacent to node $\sfv$. Based on this combinatorial graph $(\nodes,\edges)$, we construct a metric graph $\Mgraph$ by fixing an orientation for each edge $\sfe\in\edges$ and then associating this oriented edge with a finite line segment $[0,\ell^\sfe] \subset \R$, for some edge length $\ell^\sfe>0$. \\

To fix notations, given $\sfe = \sfv\sfw$, we will call $\sfv$ \emph{tail} and $\sfw$ \emph{head} and give those the orientation
\begin{equation}\label{eq:def:normalfield}
	\normal: \nodes \times \edges \to \set*{-1,0,1} \quad\text{with}\quad \normal_{\sfv}^\sfe = \begin{cases}
		-1 , & \text{if } \sfv \text{ is the tail node of } \sfe;  \\
		0 , & \text{if } \sfv \not\in \sfe  ;  \\
		1 , & \text{if } \sfv \text{ is the head node of } \sfe. 
	\end{cases}
\end{equation}
Note that the orientation of edges is used to determine the flow direction later on, with a positive sign of the velocity if the flow goes from head to tail, and a negative sign otherwise. \\

We associate the edge length $\ell^\sfe$ to each $\sfe \in \edges$ via the function
\begin{equation}\label{e:def:metric_edges}
	\ell: \edges\to (0,\infty), \qquad\text{giving rise to the set}\qquad  \Medges = \bigsqcup_{\sfe\in \edges} \sqr{0, \ell^\sfe} \,. 
\end{equation}
Here, and in the following, $\bigsqcup_{\sfe \in \edges} \sqr{0,\ell^\sfe}$ denotes the disjoint union, also called coproduct, understood as the set of ordered pairs $\bigcup_{e\in \edges} \crl{(e,[0,\ell^e])}$.
The \textbf{metric graph} is now the triplet $\Mgraph=\rnd{\nodes,\edges,\Medges}$, where in our application, $\nodes$ represents the set of pipe junctions and pipe endpoints, while $\edges$ are the pipes of the gas network, which we assume to be finite and connected. \\

\begin{definition}[1D distance on metric graph] \label{def:graph_dist_definition}
    For $x, y \in \sfG$, we define the local distance as
    \begin{equation*}
        d_{\text{loc}} \rnd{x, y} := \left\{\begin{array}{ll}
            \abs{x - y}, & \quad \text{if} ~ x, y \in \sqr{0, \ell^\sfe}, ~ \text{for the same edge} ~ \sfe \in \edges \\
            +\infty, & \quad \text{otherwise}
        \end{array}\right. ,
    \end{equation*}
    and extend it to the whole metric graph $\sfG$ with edges $\edges$ as
    \begin{equation} \label{def:graph_dist} \tag{GD}
        d \rnd{x, y} := \min_{\emptyset \nsubseteq \crl{i_1, i_2, \dots, i_n} \subseteq \crl{1, 2, \dots, k}} \crlm{\sum_{i = i_1}^{i_n} d_{\text{loc}} \rnd{x_i, y_i}} {\begin{array}{l}
            x_{i_1} = x, y_{i_n} = y, \\
            \forall j \in \crl{1, 2, \dots, n - 1}: \\
            y_{i_j}, x_{i_{j + 1}} \in \nodes, \, y_{i_j} \sim x_{i_{j + 1}}
        \end{array}} .
    \end{equation}
    The metric graph distance corresponds to the shortest path from $x \in \sfe_{i_1}$ to $y \in \sfe_{i_n}$ along the edges $\sfe_{i_2}, \sfe_{i_3}, \dots, \sfe_{i_{n - 1}}$.
\end{definition}

Note that with the results from \cite{mugnolo2019actually}, we immediately obtain that $\rnd{\sfG, d}$ is a geodesic space. Furthermore, the previously introduced Wasserstein distance is hence also well-defined in the metric graph setting, and by \cite{erbar2021gradient}, we get that $\rnd{\Prb{\sfG}, W_2}$ is a geodesic space as well. For general $p \geq 1$, in \cite{lisini2007characterization}, it was shown that $\rnd{\Prb{\sfG}, W_p}$ is a geodesic space.

\subsection{Continuity equations}
To obtain a dynamic formulation of optimal transport, we first need to introduce a notion of the continuity equation, generalizing \eqref{eq:cont_intro} to metric graphs. Note that instead of using a velocity vector field, we introduce the momentum (also known as mass flux) $J_t=\rho_t v_t$. This has the advantage of rendering the continuity equation linear in $\rho_t$ and $J_t$, and resulting in a convex optimization problem in (\ref{eq:dynamic_action}). \\

We will separately discuss the case of mass storage on nodes, and the case with classical Kirchhoff conditions. We will closely follow the definitions and notations of \cite{Heinze2024}, but also see \cite{burger2023dynamic,Fazeny_2025} for properties of the Wasserstein metric with node mass storage and \cite{erbar2021gradient} for the version with classical Kirchhoff conditions.

\subsubsection{Mass storage on nodes}
%
{To arrive at a continuity equation structure on $\Mgraph$, we start with the definition of the continuity equation on the metric edges $\Medges$, for which we specify a family of time-dependent (and in the case of $j$ signed) measures $\rho: [0,1]\to \calM_{\geq 0}(\Medges), ~ \rho\rnd{t} \eqqcolon \rho_t$ and fluxes $j: [0,1]\to \calM(\Medges), ~ j\rnd{t} \eqqcolon j_t$. Here, we write the shortcut notation $\rho_t \in \calM(\Medges)$ for $\rho_t = \crl{\rho_t^\sfe\in \calM\rnd{\sqr{0,\ell^\sfe}}}_{\sfe\in \edges}$, and $j_t \in \calM(\Medges)$ for $j_t = \crl{j^e\in \calM\rnd{\sqr{0,\ell^e}}}_{e\in \edges}$ and we call the pair $\rnd{\rho, j}$ a solution of the continuity equation if ut holds true that
\begin{equation}\label{e:def:CE_Medges}
    \forall t \in \rnd{0, 1}, \, x \in \rnd{0, \ell^\sfe}: \quad \partial_t \rho_t^\sfe\rnd{x} + \partial_x j_t^\sfe\rnd{x} = 0, \qquad \text{for all edges} ~ \sfe \in \edges,
\end{equation}
where for all $\sfe \in \edges$ we have measures $\rho^\sfe$ and $j^\sfe$ from the dual space of the continuous functions defined on $\sqr{0, \ell^\sfe}$. Note that we do not impose any continuity or differentiability constraints at the nodes.} \\

To obtain a mass preserving evolution, we complement the continuity equation with a local Kirchhoff condition, which receives the mass leaving from the metric edges due to the net-balance of the normal fluxes of $\set{j^\sfe}_{\sfe \in \edges}$ evaluated in the head- and tail nodes of the edges, and which acts as a reservoir for each of the metric edges.
{Since, in general, $j_t \in \calM(\Medges)$ does not posses a normal flux, we introduce a new set of variables $\bar\jmath^\sfe_\sfv : [0,1]\to \R$ for any $(\sfv,\sfe)\in \nodes\times \edges$, which are coupled in a weak sense by imposing the following integration by parts formula on the metric edges: For any $j^\sfe_t\in \calM\rnd{\sqr{0,\ell^\sfe}}$ with smooth density $j^\sfe_t(x)\;\mathrm{d}{x}$, we impose the identity
\begin{equation}\label{eq:flux:IntByParts}
	\int_0^{\ell^\sfe} \varphi^\sfe(x) \partial_x j_t^\sfe(x)\;\mathrm{d}{x} = \sum_{\sfv \in \nodes(\sfe)} \varphi^\sfe|_{\sfv} \, \bar\jmath^\sfe_{\sfv, t} - \int_0^{\ell^\sfe} \partial_x \varphi^\sfe \;\mathrm{d} j^\sfe_t.
\end{equation}
Here, we write $\nodes\rnd{\sfe}$ for the set of the head- and tail node of a given edge $\sfe$ and $\varphi: \Medges \to \R$ denotes a test function, which is continuous across all nodes and continuously differentiable within the respective edge.} \\

The net-balance is stored in the node density $\gamma_\sfv$ for each node $\sfv\in \nodes$ and the \emph{local Kirchhoff law} becomes
\begin{equation}\label{e:def:localKirchhoff}
 \forall \sfv \in \nodes  : \qquad \partial_t \gamma_{\sfv, t} = \sum_{\sfe\in\edges(\sfv)} \bar\jmath^\sfe_{\sfv, t}  \qquad\text{ for a.e. } t \in [0,1]. 
\end{equation}
We define the space of probability measures on $\Mgraph$ as tuples $\mu_t = (\gamma_t,\rho_t) \in \calP(\Mgraph) \subseteq \calM_{\geq 0}(\nodes)\times \calM_{\geq 0}(\Medges)$ such that $\gamma_t(\nodes)+ \rho_t(\Medges)=1$. We are now in a position to introduce our notion of the continuity equation on a metric graph with node mass storage.
\begin{definition}[Continuity equation, \cite{Heinze2024}]\label{def:bCE}
    {A family of curves $\mu=\bra*{\gamma,\rho} : [0,1] \to \calP(\Mgraph)$ with a flux pair $\rmj= (\bar\jmath,j)  : [0,1] \to \calM((\nodes \times \edges)\times \Medges)$ satisfies the \emph{continuity equation on the metric graph with node reservoirs} if ~\eqref{e:def:CE_Medges} and~\eqref{e:def:localKirchhoff} hold, and $\rnd{\mu, \rmj}$ solves the abstract continuity equation
	\begin{equation}\label{eq:def:bCE}
		\partial_t \mu_t + \bdiv \rmj_t = 0,
	\end{equation}
    where for a.e. $t \in \rnd{0, 1}$, the measures $\mu$ and $j$ are form the dual space of the continuous functions defined on $\nodes \times \Medges$. For such a solution we then write $\rnd{\mu, \rmj} \in \CE_\nodes$.} In terms of duality, for $\Phi=(\phi,\varphi)\in C^1(\nodes \times \Medges)$ the abstract version of the continuity equation is defined by
    \begin{equation}\label{eq:cont_vertex_weak}
	\begin{aligned}
		& \pderiv{}{t} \pra[\bigg]{ \sum_{\sfv \in \nodes} \phi_\sfv \gamma_{\sfv, t} + \sum_{\sfe \in\edges} \int_0^{\ell^\sfe} \mkern-8mu \varphi^\sfe(x) \dd \rho^\sfe_t} = \\
        & \sum_{\sfv \in \nodes} \sum_{\sfe \in\edges(\sfv)} \bra*{\phi_\sfv - \varphi^\sfe|_\sfv} \bar\jmath^\sfe_{\sfv, t} + \sum_{\sfe\in\edges} \int_0^{\ell^\sfe}\mkern-8mu \partial_x \varphi^\sfe \dd j^\sfe_t  \,
    \end{aligned}
	\end{equation}
	for a.e. $t\in \rnd{0,1}$. Furthermore, we denote by $\CE_{\nodes}\rnd{\mu^0,\mu^1}$ the set of all weak solutions $\rnd{\mu,j}$ to \eqref{eq:def:bCE}, such that given initial- and final conditions $\mu_0 = \mu^0$ and $\mu_1=\mu^1$ are fulfilled.
\end{definition}
Among the solutions to~\eqref{eq:def:bCE} with bounded flux, we can find an absolutely continuous representative. \\

\begin{lemma}[Well-posedness of $\CE$, \cite{Heinze2024}]\label{lem:AC:representative}
    {Let $(\mu,\rmj)\in\CE$ with
	\begin{equation}\label{eq:flux:div:integrable}
		\int_0^1\pra[\bigg]{ \sum_{\sfv \in \nodes} \sum_{\sfe\in \edges(\sfv)} \abs{\bar\jmath^\sfe_{\sfv, t}} +  \sum_{\sfe\in \edges} \int_0^{\ell^\sfe} \dd \abs{j^\sfe_t}} \dd t < \infty \,,
	\end{equation}
	then there exists $\rnd{\tilde\mu,\tilde \rmjmath}\in\CE$ such that $\rnd{\tilde \mu_t}_{t \in \sqr{0, 1}} \subset \Prb{\Mgraph}$ is an absolutely continuous curve equipped with the narrow topology of measures. 
    This means that for all $\Phi=(\phi,\varphi) \in C(\nodes \times \Medges)$ the map
	\begin{equation*}
		t \mapsto \skp{\varphi,\tilde\mu_t}_{\sfV\times \sfL} \coloneqq \sum_{\sfv\in\nodes} \phi_\sfv \tilde\gamma_{\sfv, t} + \sum_{\sfe \in \edges} \int_0^{\ell^\sfe} \varphi^\sfe \dd \tilde\rho_t^\sfe
	\end{equation*}
    is absolutely continuous.}
\end{lemma}
\begin{proof}
	The proof is completely analogous to \cite[Lemma 8.1.2]{ambrosio2005gradient} for the metric edge parts and \cite[Lemma 3.1]{Erb14} for the edge-node-transition parts. 
\end{proof}


\subsubsection{Kirchhoff case}
In the case of classical Kirchhoff's law at the nodes, we make use of the continuity equation structure introduced in \cite{erbar2021gradient}, which is set on the quotient space
    \begin{equation*}
        \sfG \coloneqq \Medges / \sim,
    \end{equation*}
where the equivalence relation $\sim$ identifies the head and tail nodes of different intervals corresponding to the same node. This means that we consider a situation where the edges are glued together at the nodes to form one connected space domain. \\

The main difference compared to the previous section is the use of test functions that are continuous over the complete graph. Thus, the boundary terms which were present in \eqref{eq:cont_vertex_weak} disappear in the weak formulation, encoding a Kirchhoff condition on the fluxes.
\begin{definition}[Continuity equation (Kirchhoff) \cite{erbar2021gradient}]\label{def:CE:Kirchhoff}
    A family of curves $\rnd{\rho_t, J_t}_{t \in [0, 1]}$ consisting of probability measures $\rho_t \in \Prb{\sfG}$ and signed measures $J_t \in \Mea{\Medges}$,
    such that for all Borel sets $A$ the map $t \mapsto \rnd{\rho_t\rnd{A}, J_t\rnd{A}}$ is Borel measurable, is called a weak solution of \emph{continuity equation on the metric graph with Kirchhoff condition} if for all test functions $\psi \in C^1\rnd{\Medges} \cap C\rnd{\sfG}$ and a.e. $t \in \rnd{0, 1}$ it holds
    \begin{align}\label{e:def:bCE:weak_Kirchhoff}
    \frac{\mathrm d}{\mathrm d t} \int_\sfG \psi \dd \rho_t = \int_\Medges
    \nabla_x \psi \cdot \dd J_t,
    \end{align}
    where we assume
    \begin{enumerate}
        \item[(1)] $\forall \psi: \, t \mapsto \int_\sfG \psi \dd \rho_t$ is absolutely continuous,
        \item[(2)] $\int_{\rnd{0, 1}} \abs{J_t}\rnd{\Medges} 
        \dd t < \infty$.
    \end{enumerate}
    We denote by $\CE_{K}(\rho^0,\rho^1)$ the set of all weak solutions $\rnd{\rho,J}$ to \eqref{e:def:bCE:weak_Kirchhoff} such that $\rho_0 = \rho^0$ and $\rho_1=\rho^1$.
\end{definition}
\begin{remark}
{The Kirchhoff condition is implicitly contained in the weak formulation \eqref{e:def:bCE:weak_Kirchhoff}. Indeed, if the flux would be sufficiently regular to allow for integration by parts, due to the continuity of the test function across nodes, we would obtain that
$$
\sum_{e\in\edges(v)} \bar\jmath^\sfe_{\sfv} = \sum_{\sfe\in\edges(\sfv)} j_t^\sfe\rnd{\sfv} \normal_{\sfv}^\sfe = 0\qquad \text{for all} ~ \sfv \in\nodes
$$
must hold.}
\end{remark}

\subsection{Action functionals and dynamic distances}

We are now able to define dynamic transport distances based on the minimization of action functionals over curves which satisfy the respective continuity equations.

\subsubsection{Mass storage on nodes}

In view of our application, mass storage on nodes may have different interpretations. 
{It is for example possible to use mass storage on nodes to model gas inflow and outflow for the network. To this end, one differentiates between boundary nodes $\partial\nodes$, where gas may enter or exit the network (as supply or demand), and interior nodes $\mathring{\nodes}$. Then, the set of interior nodes is defined as 
\begin{equation*}
    \mathring{\nodes} = \nodes \backslash \partial \nodes,
\end{equation*}
which directly implies $\partial \nodes \cup \mathring{\nodes} = \nodes$ and $\partial \nodes \cap \mathring{\nodes} = \emptyset$ so that each node is either a boundary node or an interior node.} \\

For the set of boundary nodes $\partial \nodes$, it is possible to differentiate further between source nodes $\partial^+ \nodes$, which supply the network with gas, and sink nodes $\partial^- \nodes$, at which gas is taken out of the network to meet given demands. 
The inflow and outflow of gas into and out from the network can then be modeled via time-dependent boundary conditions at $\partial\nodes$. For the sake of readability, however, we will limit our exposition to the case without boundary nodes and refer to \cite{Fazeny_2025} for details on those. \\

For given initial- and final mass distributions  $\mu^0 := \left(\gamma^0,\rho^0\right)$ and $\mu^1 := \left(\gamma^1,\rho^1\right)$, we define the dynamical formulation of the Wasserstein distance as
\begin{equation}\label{eq:W_V}
        \begin{aligned}
            & W_{\nodes,p} \rnd{\mu^0, \mu^1} = \\
            & \min_{\rnd{\mu, j} \in \CE_{\nodes}\rnd{\mu^0,\mu^1}} \Bigg\{\int_0^1 \sum_{\sfe \in \edges} \left(\int_0^{\ell^\sfe} h\left(j_t^\sfe, \rho_t^\sfe\right) \dd \eta\right) + \sum_{\sfv \in \nodes} h\left(\sum_{\sfe \in\edges(\sfv)} \bar\jmath^\sfe_{\sfv, t}, \gamma_{\sfv, t}\right) \dd t\Bigg\}^{1/p},
        \end{aligned}
    \end{equation}
where the \emph{perspective function} $h:\mathbb{R} \times \mathbb{R} \longrightarrow \mathbb{R}_{\geq 0} $ is given as 
\begin{equation*}
    \label{eqn:helper}
\left(a, b\right) \mapsto h\left(a, b\right) = \left\{ \begin{array}{ll}
        \frac{\left\lvert a \right\rvert^p}{b^{p - 1}}, \quad & \text{if} ~ b > 0, \vspace{0.1cm}\\
        0, \quad & \text{if} ~ b = a = 0,\\
        +\infty, \quad & \text{if} ~ b < 0 ~ \vee ~ b = 0, \, a \neq 0,
    \end{array} \right. 
\end{equation*}
and it is responsible for penalizing negative densities and non-zero flux in the case of zero densities. Here, $\eta$ denotes a reference measure and, abusing notation, we denote densities of $j^\sfe$ and $\rho^\sfe$ with respect to $\eta$ still by the same symbols. \\

While \eqref{eq:W_V} was introduced in \cite{Fazeny_2025}, the gradient flows studied in \cite{Heinze2024} are understood with respect to a different version of the Wasserstein distance. Indeed, in \eqref{eq:W_V}, the contributions at the nodes depend on 
$$
\sum_{\sfe \in\edges(\sfv)} \bar\jmath^\sfe_{\sfv, t}, 
$$
i.e., the net flux summing the contributions of all edges adjacent to the node $\sfv$. In contrast to that, in \cite{Heinze2024}, each contribution is considered separately which yields the metric
\begin{equation}\label{eq:W_V_Heinze}
        \begin{aligned}
            & \overline W_{\nodes,p} (\mu^0, \mu^1) = \\
            & \min_{(\mu,j)\in \CE_{\nodes}(\mu^0,\mu^1)} \Bigg\{\int_0^1 \sum_{\sfe \in \edges} \left(\int_0^{\ell^\sfe} h\left(j^\sfe_t, \rho^\sfe_t\right) \, \mathrm d \eta\right) + \sum_{\sfv \in \nodes}\sum_{\sfe \in\edges(\sfv)} h\left(\bar\jmath^\sfe_{\sfv, t}, \gamma_{\sfv, t}\right) \, \mathrm d t\Bigg\}^{1/p}.
        \end{aligned}
    \end{equation}

\subsubsection{Kirchhoff case}

In the classical Kirchhoff case, we are considering curves $(\rho,J) \in \CE_{K}\left(\rho^0,\rho^1\right)$ for given initial- and final densities $\rho^0, \rho^1 \in \calP(\sfG)$. The Wasserstein distance is then given by
\begin{equation*}
    W_{K,p}(\rho^0, \rho^1) = \min_{\rnd{\rho, \rho v} \in \CE_K\rnd{\rho^0, \rho^1}} \crl{\int_0^1 \sum_{\sfe \in \edges} \int_0^{\ell^\sfe} \abs{v_t}^p \dd \rho_t \dd t}^{1/p},
\end{equation*}
with $\CE_K\rnd{\rho^0, \rho^1}$ as in Definition~\ref{def:CE:Kirchhoff}.
In order for the right hand side of $J_t = v_t \, \rho_t$ to be well-defined, we need a characterization of absolutely continuous curves in the $p$-Wasserstein space. This is given by a theorem proven in \cite{weigand2025pdiffusion}, which requires the following two definitions.
\begin{definition}[Absolutely continuous curve \cite{ambrosio2005gradient}]
    For $p \geq 1$, a curve $\zeta: \sqr{0, 1} \rightarrow \Prb{\sfG}$  is called $p$-absolutely continuous in the $p$-Wasserstein space, if there exists a function $m \in L^p\rnd{0, 1}$ such that for all $0 < s \leq t < 1$ it holds true that
    \begin{equation} \label{abs_cont}
        W_p\rnd{\zeta\rnd{s},\zeta\rnd{t}} \leq \int_s^t m\rnd{r} \dd r.
    \end{equation}
\end{definition}
\begin{definition}[Metric derivative \cite{ambrosio2005gradient}]
    For every $p$-absolutely continuous curve (in the $p$-Wasserstein space) $\zeta: \sqr{0, 1} \rightarrow \Prb{\sfG}$, its metric derivative is defined as
    \begin{align*}
        \abs{\dot{\zeta}}\rnd{t} \coloneq \lim_{s \rightarrow t} \frac{d\rnd{\zeta\rnd{s}, \zeta\rnd{t}}}{\abs{s - t}}
    \end{align*}
    for $t \in \rnd{0, 1}$ and it exists for a.e. such $t$, where $d$ is the shortest path distance on the metric graph, see~(\ref{def:graph_dist}).
\end{definition}
In \cite{ambrosio2005gradient}, it is proven that the metric derivative of a $p$-absolutely continuous curve exists for a.e. $t \in \rnd{0,1}$ and that $\abs{\dot{\zeta}} \in L^p\rnd{0, 1}$. Furthermore, $\abs{\dot{\zeta}}$ satisfies (\ref{abs_cont}) chosen as $m$, and $\abs{\dot{\zeta}}$ is minimal, meaning that for any other $m \in L^p\rnd{0, 1}$ satisfying the equation, it holds true that $\abs{\dot\zeta}\rnd{t} \leq m\rnd{t}$ for a.e. $t \in \rnd{0, 1}$.
\begin{theorem}[Absolutely continuous curves] \label{theorem:abs_cont}
    \begin{itemize}
        \item[(i)] If $\rho: \sqr{0, 1} \to \Prb{\sfG}$ is an absolutely continuous curve in $\rnd{\Prb{\sfG}, W_p}$ 
        then there exists, for a.e. $t \in \rnd{0,1}$, a vector field $v_t \in L^p\rnd{\rho_t}$ such that $\int_\sfG \norm{v_t} \dd \rho_t \leq \abs{\dot{\rho}}\rnd{t}$ and $\rnd{\rho, \rho v}$ is a weak solution to the continuity equation, as defined in \ref{def:CE:Kirchhoff}. \vspace{0.25cm}
        \item[(ii)]\ If $\rnd{\rho, \rho v}$
        is a weak solution to the continuity equation, as specified in \ref{def:CE:Kirchhoff} and \[
        \int_{\rnd{0,1}} \int_\sfG \norm{v_t} \dd \rho_t \dd t < \infty,
        \]
        \newline then $\rho: \sqr{0, 1} \to \Prb{\sfG}$ is an absolutely continuous curve in $\rnd{\Prb{\sfG}, W_p}$ and for almost every $t \in \rnd{0, 1}$, we obtain $\abs{\dot{\rho}\rnd{t}} \leq \int_\sfG \norm{\vel_t} \dd \rho_t$.
    \end{itemize}
\end{theorem}
This result was already proven for $p\! =\! 2$ in \cite{erbar2021gradient}, and it was generalized to $p\! \geq\! 1$ in \cite{weigand2025pdiffusion}. It is an important result, as Theorem \ref{theorem:abs_cont} (i) shows that every $W_p$-absolutely continuous curve $\rho_t$ can be realized as a physical mass flow with an $L^p$-finite velocity field $v_t$ and mass conservation, so $\rho_t$ encodes a feasible optimal transport of the gas through the network. Moreover, Theorem \ref{theorem:abs_cont} (ii) shows the opposite direction, meaning that having a mass flow with finite $p$-energy, which also satisfies the continuity equation, then the corresponding mass trajectory $\rho_t$ is an $W_p$-absolutely continuous curve, thus any weak solution of \ref{e:def:bCE:weak_Kirchhoff} with an integrable velocity $v_t$ is a feasible optimal transport. Regarding the dynamic formulation of the $p$-Wasserstein distance, this theorem ensures that the mass flux $J_t$ from a weak solution of the continuity equation can be decomposed into $J_t = \rho_t \, v_t$.

\subsection{Gradient flows on metric graphs}

Since the introduction of Otto-Wasserstein gradient flow, \cite{JordanKinderlehrerOtto1997,JordanKinderlehrerOtto1998, Otto2001}, many equations were recognized as gradient flows in Wasserstein-type metrics. We briefly review some results for equations set on metric graphs.

\subsubsection{Mass storage on nodes}

So far, gradient flows in the setting of mass storage on nodes have not been studied in relation to gas networks. Instead, linear drift-diffusion equations were considered in \cite{Heinze2024}. 
To introduce the system, we fix a positive reference probability measure $(\omega,\pi)\in \calP_+(\Mgraph)$, where we assume $\pi$ to have a density given by $\pi(\mathrm{dx})= e^{-P(x)} \mathrm{dx}$ for some potential $P: \Medges\to \R$, such that $P^\sfe : \sqr{0,\ell^\sfe}\to \R$ is Lipschitz for all $\sfe\in \edges$. With this, we consider for a.e. $t \in (0,1)$, the system
\begin{subequations}\label{eq:system_intro_linear}
\begin{align}
    \label{eq:edge_intro_linear}
    & \partial_t \rho^\sfe = \diffedge^\sfe \partial_x\cdot\bra*{\partial_x \rho^\sfe +\rho^\sfe \partial_x P^\sfe },  &&\text{on } [0,\ell^e] \text{ for every }  \sfe \in \edges, \\
    \label{eq:Kirchhoff_intro_linear}
    & - \diffedge^\sfe \left.\Big(\partial_x \rho^\sfe + \rho^\sfe \partial_x P^\sfe\right)\Big|_{\sfv}\cdot \sfn_{\sfv}^\sfe =r(\sfe,\sfv) \, \left.\rho^\sfe\right|_{\sfv} - r(\sfv,\sfe) \gamma_\sfv ,  && \text{for all } \sfe\in \edges, \, \sfv\in \nodes,\\
    \label{eq:vertex_intro_linear}
    & \partial_t\gamma_\sfv = \sum_{\sfe\in \edges(\sfv)} \Big(r(\sfe,\sfv) \, \left.\rho^\sfe\right|_{\sfv} - r(\sfv,\sfe) \gamma_\sfv\Big),  && \text{for all } \sfv\in \nodes \,.
\end{align}
\end{subequations}
We set $\left.\rho^\sfe\right|_{\sfv}= \rho^\sfe(0)$ if $\sfe=\sfv\sfw$, and we set $\left.\rho^\sfe\right|_\sfv = \rho^\sfe(\ell^\sfe)$ if $\sfe=\sfw\sfv$ for some $\sfw\in \nodes$.
Moreover, we introduced the parameter $\diffedge^\sfe>0$ as the diffusion constant on the metric edge $\sfe$, $r(\sfe,\sfv)$ as the jump rate from the endpoint of the metric edge associated with $\sfe$ to an adjacent node $\sfv$, and $r(\sfv,\sfe)$ for the reverse jump rate from the node $\sfv$ to the adjacent edge $\sfe$. \\

The main result of \cite{Heinze2024} is that system~\eqref{eq:system_intro_linear} admits a gradient flow formulation, provided that the jump rates satisfy the \emph{detailed balance condition}, with respect to the reference measure $(\omega,\pi)\in \calP_+(\Mgraph)$, that is it holds
\begin{equation}\label{eq:DBC:r}
	r(\sfe,\sfv) 
    \sqrt{\frac{\pi^\sfe\smash[t]{|_\sfv}}{\omega_\sfv}}
    = r(\sfv,\sfe) 
    \sqrt{\frac{\omega_\sfv}{\pi^\sfe|_\sfv}} \eqqcolon \scrk^\sfe_{\sfv}
    \qquad \text{for all } \sfe\in \edges, \, \sfv\in \nodes \,.
\end{equation}
The driving entropy functional is given by 
\begin{align}\label{eq:def:free_energy}
	\calE(\mu) &\coloneqq \calE_{\Medges}(\rho) + \calE_{\nodes}(\gamma) \coloneqq \sum_{\sfe\in\edges}\calH\left(\rho^\sfe \middle| \pi^\sfe\right)+
	\sum_{\sfv\in\nodes} 
	\calH\left(\gamma_\sfv \middle|\omega_\sfv\right) \,,
\end{align}
where the relative entropy $\calH$ is defined as
\begin{equation}\label{eq:def:RelEnt}
	\calH\left(\mu \middle| \nu\right) \coloneqq \begin{cases}
		\int \eta\bra[\big]{\pderiv{\mu}{\nu}} \dd \nu ,& \text{if } \mu \ll \nu \,, \\
		+\infty , &\text{else,}
	\end{cases}
\end{equation}
where $\eta(r)\coloneqq r\log r -r +1$ for $r\geq 0$.
Here and in the following, we use the convention that $0\cdot \log 0 =0$.

\begin{remark}
As the above system~\eqref{eq:system_intro_linear} is a combination of diffusion on the metric edges and a reaction-like mechanism for the mass exchange between edges and nodes, the gradient flow structure also consists of two parts. While the edge dynamics can be captured by quadratic action functionals as introduced above, the reaction part requires a non-linear $\cosh$-type functional. Thus, even on a formal level, the combination of quadratic diffusion terms with non-linear reaction terms does not yield a metric structure, and instead a JKO approach as in \eqref{eq:JKO}, one uses an $\mathcal R^*$–$\mathcal R$ formulation, see also \cite{LieroMielkePeletierRenger2017} for details.
\end{remark}

\subsubsection{Kirchhoff case}

In \cite{Fazeny_2025}, it was shown that the \eqref{eqn:ISO3} model is a $W_{K,3}$-gradient flow with respect to the energy functional
\begin{equation*}
    \mathcal{E} \rnd{\rho} \coloneqq \sum_{\sfe \in \edges} \mathcal{E}^\sfe\rnd{\rho^\sfe} = \sum_{\sfe \in \edges} \int_0^{\ell^\sfe} F^\sfe\rnd{\rho^\sfe} + \frac{2 D^\sfe g}{\lambda^\sfe} \sin\rnd{\omega^\sfe} \rho^\sfe x + d^\sfe \rho^\sfe \dd x.
\end{equation*}
Here, the entropy density $F^\sfe$ for each edge $\sfe$ is defined based on the pressure-entropy relation
\begin{equation*}
    {F^\sfe}''\rnd{\rho^\sfe} = \frac{2 D^\sfe}{\lambda^\sfe} \frac{p'\rnd{\rho^\sfe}}{\rho^\sfe}
\end{equation*}
together with $F^\sfe\rnd{0} \coloneqq 0$. Moreover, the integration constants $d_e$ arise from the fact that even though the global mass within the network is conserved, the mass per edge can change between $\rho_0$ and $\rho_1$. \\

In order to derive $p$-Wasserstein gradient flows with the defined energy $\mathcal{E}$, the optimality conditions of the \textit{minimizing movement scheme}
\begin{equation}\label{eq:JKO}
    \rho_{\rnd{i + 1} \tau} = \text{arg} \min_{\rho \in \Prb{\sfG}} \crl{\frac{1}{p \,\tau^{p - 1}} W_{K,p}^p\rnd{\rho, \rho_{i \tau}} + \mathcal{E}\rnd{\rho}},
\end{equation}
starting at the initial network state $\rho_0 \coloneqq \rho_0$, for step size $\tau > 0$ are analyzed. It is then proven that for $\tau \rightarrow 0$, an appropriate time interpolation of the minimizers converges to a solution of the respective $p$-Wasserstein gradient flow. \\

The minimizing movement scheme provides a variational time discretization of the gradient flow associated with the energy functional $\mathcal{E}$ in the $p$-Wasserstein space, where in each time step the new $\rho_{\rnd{i + 1} \tau}$ is chosen such that on the one hand it is ``close enough'' to the previous network state $\rho_{i \tau}$ (minimizing the transport cost), and on the other hand such that the new network state is a low energy state (as we aim at minimizing the energy $\mathcal{E}$). \\

Taking $\tau \rightarrow 0$ in the optimality conditions, then the resulting equation system of the $p$-Wasserstein gradient flow is given by
\begin{align*}
    \partial_t \rho^\sfe + \partial_x j^\sfe & = 0, \\
    j^\sfe \, \abs{j^\sfe}^{p - 2} & = - \rnd{\rho^\sfe}^{p - 1} \, \partial_x \rnd{\mathcal{E}'\rnd{\rho^\sfe}}.
\end{align*}
Plugging in our specific choice for the energy functional $\mathcal{E}$ as well as $p = 3$ results in \eqref{eqn:ISO3}, expressed by
\begin{align*}
    \partial_t \rho^\sfe + \partial_x j^\sfe & = 0, \\
    j^\sfe \, \abs{j^\sfe}^{p - 2} & = - \rnd{\rho^\sfe}^2 \, \rnd{\frac{2 D^\sfe}{\lambda^\sfe \rho^\sfe} \partial_x p^\sfe\rnd{\rho^\sfe} + \frac{2 D^\sfe g}{\lambda^\sfe} \sin\rnd{\omega^\sfe}}.
\end{align*}

For the same Wasserstein distance, yet in the case $p=2$, the authors of \cite{erbar2021gradient} rigorously proved that the drift--diffusion--aggregation equation
\begin{align*}
\partial_t \rho=\Delta \rho+\nabla \cdot\Big(\rho\rnd{\nabla V+\nabla W\sqr{\rho}}\Big)
\end{align*}
with non-local interaction part
$$
W\sqr{\mu}(x):=\int_{\sfG} W(x, y) \dd \mu(y)
$$
is a gradient flow with respect to the metric $W_{K,2}$.
In this case, the energy functional is given as a sum of the logarithmic entropy and the non-local part
%
%
%
$$
\mathcal{W}(\rho):=\frac{1}{2} \int_{\sfG \times \sfG} W(x, y) \dd \rho(x) \mathrm{d} \rho(y),
$$
where $W: \sfG \times \sfG \rightarrow \mathbb{R}$ is symmetric and Lipschitz continuous. 

%% file: 3_3D.tex
\section{Zero-width limit of gas networks}\label{sec:3d_to_1d}



To prove our main results we need to introduce some concepts that can be found in most textbooks on the subject of optimal transport, such as \cite{villani2008optimal,Santambrogio2015}. The major concepts introduced below were already motivated in the introduction. 

\begin{definition}[Cyclical monotonicity]
\label{def: Cycl}
Let $X$,$Y$ be arbitrary
sets, and take a cost function $c: X \times Y \to (-\infty, +\infty]$ . A subset $\Gamma \subset X \times Y$
is said to be $c$-cyclically monotone if, for any $N \in \mathbb{N}$, and any family
$(x_1,y_1),\ldots,(x_N,y_N)$ of points in $\Gamma$, the inequality
 \[
 \sum_{i=1}^N c(x_i, y_i) \le  \sum_{i=1}^N c(x_i, y_{i+1}).
 \] 
holds, using the convention $y_{N+1} = y_1$. A transference plan or coupling is said to be
$c$-cyclically monotone if it is concentrated on a $c$-cyclically monotone set.
\end{definition}

\begin{definition}[$c$-concavity]
A function $\phi : Y \to \mathbb{R}\cup \{-\infty\}$ is called $c$-concave if it
is not identically $-\infty$ and there exists $\psi : X \to \mathbb{R}\cup \{+\infty\}$,
such that
\[
\phi(y) = \inf_{x\in X} \Big\{c(x,y)- \psi(x)\Big\}.
\]
\end{definition}

\begin{theorem}[Kantorovich Duality]
\label{thm: KantDual}
    Let $X$ and $Y$ be Polish spaces, let $\mu \in \mathcal{P}(X)$ and $\nu \in \mathcal{P}(Y)$, and let $c: X \times Y \to \R_+ \cup \{ +\infty\}$ be lower semicontinuous. Then the following holds:
    \begin{enumerate}
        \item Kantorovich duality:
        \begin{equation} \label{eq: dual}
            \operatorname{OT}(\mu, \nu, c):=\inf_{\pi \in \Pi(\mu, \nu)} \crl{\int c \, \mathrm{d}\pi} = \sup_ {\substack{\phi \in L^1(\mu), \, \psi \in L^1(\nu), \\ \phi(x) \, + \,  \psi(y) \, \le \, c(x,y)}} \crl{\int \phi \, \mathrm{d}\mu + \int \psi \, \mathrm{d}\nu}.
        \end{equation}
        \item If c is real-valued and the optimal cost $\operatorname{OT}(\mu, \nu, c)$ is finite, then there exists a $\pi$-measurable $c$-cyclically monotone set $\Gamma \subset X \times Y $ (closed if $c$ is continuous) such that for any $\pi \in \Pi(\mu, \nu)$ the following statements are equivalent: \\
        (a) $\pi$ is optimal; \\
        (b) $\pi$ is c-cyclically monotone; \\
        (c) $\pi$ is concentrated on $\Gamma$; \\
        (d) there exists a $c$-concave function $\phi$ such that the equation $\phi(x) + \phi^c(y) = c(x,y)$ holds $\pi$-almost surely.
    \end{enumerate}
\end{theorem}
For a proof we refer to Theorem~5.10 in \cite{villani2008optimal}.

\subsection{3D modeling of gas network}

In the three-dimensional case, we assume that each pipe has diameter $D = 2 \eps > 0$ and thus it will be represented by a cylinder with length $\ell^\eps$ and radius $\eps$, instead of a segment.

\begin{remark}[3D metric graph]
    While the sets $\nodes$ and $\edges$ remain the same in the 3D setting, the set $\Medges$ changes to
    \begin{equation*}
        \Medges = \bigsqcup_{e\in \edges} \Be \times \sqr{0, \ell^\eps}, 
    \end{equation*}
    where $\Be \coloneqq \crlm{x \in \R^2}{\norm{x}_2 \leq \eps}$ is the two-dimensional $\eps$-disk around the origin. Depending on the angle between pipes, for two connected three-dimensional edges $e_1, e_2 \in \edges$ the spaces $\Be \times \rnd{0, \ell^{e_1}}$ and $\Be \times \rnd{0, \ell^{e_2}}$ are in general not disjoint. Therefore, alternatively the edges $e$ can also only be assigned the respective disjoint space
    \begin{equation*}
        \Be \times \sqr{0, \ell^e} \Big\backslash \Big(\bigcup_{\tilde{e} \in \edges \backslash \crl{e}} \Be \times \sqr{0, \ell^{\tilde{e}}}\Big),
    \end{equation*}
    while the remaining ``shared'' space is assigned to the corresponding vertex, so for each $\sfv \in \nodes$ we would assign
    \begin{equation*}
        \bigcup_{\substack{e_1, e_2 \in \edges: \, e_1 \neq e_2, \\
        \sfv \in e_1, \, \sfv \in e_2}} \rnd{\Be \times \sqr{0, \ell^{e_1}}} \cap \rnd{\Be \times \sqr{0, \ell^{e_2}}}.
    \end{equation*}
    For our following results this distinction is of no importance, as we will always consider the union of all three-dimensional edges embedded in $\R^3$.
\end{remark}

%% file: 4_Main_Results.tex
\subsection{Main results}

Let $\Nov \coloneqq \G \subset \R$ be the one-dimensional metric graph, which encodes the general structure of the gas network. Furthermore, we encode the three-dimensional gas network with $\mathcal{N}_\eps \coloneqq \Nov + \mathbb{B}_\eps \subset \R^3$, where the sum means Minkowski addition and $\mathbb{B}_\eps \coloneqq \crlm{x \in \R^3}{\norm{x}_2 \leq \eps}$ is the closed ball of radius $\eps$ in $\R^3$. Then, $\No$ corresponds to the 1D metric graph embedded in $\R^3$, thus its intrinsic dimension is 1. \\

For a well-defined limit $\eps \to 0$ of the densities, describing the distribution of gas in the network, we need probability measures $\rnd{\mue}_\eps, \rnd{\nue}_\eps \in \Prb{\R^3}$ with support $\Ne$, which weakly converge to probability measures $\muo, \nuo \in \Prb{\No}$. These limits $\muo$ and $\nuo$ correspond to lifted versions of probability measures $\muov, \nuov \in \Prb{\Nov}$, where the mass is allocated only along the lines of $\Nov$. This means that $\mu_0$ and $\nu_0$ can be written as
\begin{equation*}
    \mu_0 = \overline{\mu} \otimes \delta_{x_2, x_3 = 0} \in \Prb{\Ne} \quad \text{and} \quad \nu_0 = \overline{\nu} \otimes \delta_{y_2, y_3 = 0} \in \Prb{\Ne},
\end{equation*}
for the Dirac measure $\delta$ on $\R^3$, so for any $Z \subseteq \R^3$ it is defined as
\begin{equation*}
    \delta_{z_2, z_3 = 0} \rnd{Z} = \Bigg\{\begin{array}{ll}
        1 & \quad \text{if} ~ \exists z \in Z: ~ z_2 = z_3 = 0, \\
        0 & \quad \text{otherwise.}
    \end{array}
\end{equation*}

In the following, for the support of the probability measures we will use the short-hand notation $\operatorname{spt}\rnd{\mue}, \operatorname{spt}\rnd{\nue} = \Ne$.

\subsection{Convergence on the entire network.}

We first define a penalty function $\iota_\Ne$ on $\R^3$, that incorporates the geometry of the 3D gas network $\Ne$, penalizing any transport outside of the network. 
For a compact set $U\subset \R^3$, let
\[
    \iota_U(x) := \begin{cases} \ 0, \quad\quad & \text{if} ~ x \in U,\\ \ +\infty,  & \text{else}.
\end{cases}
\]
It is a standard result from convex analysis that $\iota_U(x)$ is lower-semicontinuous when $U$ is closed. \\

By $\mathcal{C}(x,y)$ we denote the set of piecewise continuously differentiable curves $\gamma: \sqr{0, 1} \rightarrow \R^3, \, t \mapsto \gamma_t$ with endpoints $x,y$, i.e., $\gamma \in  \mathcal{C}(x,y)$ implies that $\gamma_0=x$, $\gamma_1=y$ and $\dot\gamma_t := \mathrm{d} \gamma_t /  \mathrm{d}t$ is well-defined for a.e. $t \in \sqr{0, 1}$. \\

Define the cost function $\ce(x,y) : \R^3 \times\R^3 \to \R \cup \{+\infty\} $ as 
\begin{equation}
\label{eq: cost}
    \ce(x,y) := \inf_{\gamma \in \mathcal{C}(x,y)} \crl{\int_0^1 \|\dot \gamma_t\|^2 + \iota_{\mathcal{N}_\eps}(\gamma_t) \dd t}.
\end{equation}

Importantly, remark that for two points $x, y$ in the same pipe of the three-dimensional network, the infimum in (\ref{eq: cost}) is attained and the optimal path $\gamma_t$ is a straight line between $x$ and $y$. The defined cost function is particularly attractive as it combines the aforementioned property, while ensuring that the cost function is well defined between any two points of the network and thus gives rise to a geodesic-type of distance. \\

The optimal transport problem on the 3D network is then given by
\begin{equation}
\label{eq: OTeps} \tag{OT}
\operatorname{OT}\rnd{\mue, \nue, \ce} = \inf_{\pie \in \Pi(\mu_\eps, \nu_\eps)} \crl{\int_{\R^3 \times \R^3} \ce(x,y) \dd \pie(x,y)},
\end{equation}
as in \eqref{eq: dual}.

\subsubsection{Properties of the cost}
\label{sec: PropCost}
First, notice that defining cost functions based on action minimizing curves is a classical theme in optimal transport, see for instance~\cite[Chapter~7]{villani2008optimal}. The problem of interest can be thought of as plain optimal transport, where the trajectories $\gamma$ need to avoid exiting the pipes, a problem conceptually not far from optimal transport in the presence of obstacles. In this last setting for instance, \cite{CARDALIAGUET201143} relies on a similar
construction. \\

In a more computational context, \cite{pooladian2024neural} proposed to study a similar problem where $\iota_U(x)$ is replaced by a smooth function acting as a barrier. A smooth penalty function has indeed additional computational benefits, however it does not strictly enforce the condition of staying within the network, which could lead to $\gamma$ ``cutting the corner'' at pipe intersections. \\

The introduced cost function relies on curves minimizing the \emph{action} 
$$\mathcal{A}(\gamma) :=\int_0^1 \|\dot \gamma_t\|^2 + \iota_{\mathcal{N}_\eps}(\gamma_t) \dd t,$$ where the integrand is often called a \emph{Lagrangian}.  
This framework is covered in detail in \cite[Chapter~7]{villani2008optimal}. It can be easily verified that $\|v\|^2 + \iota_{\mathcal{N}_\eps}(x)$
is convex and superlinear in $v$. 

\begin{lemma}
For any $x \in \operatorname{int}(\mathcal{N}_\eps)$, the interior of $\mathcal{N}_\eps$, assume that there is a unique geodesic $\gamma_t^\star$, i.e., a unique minimizer in the definition of $\ce(x,y)$, joining $x$ to a given $y \in \mathcal{N}_\eps$. Then the cost $\ce(x,y)$ is $C^1$ in $x$ and 
\[
\nabla_x \ce(x,y) = - 2 \dot \gamma_t^\star \big\vert_{t=0}.
\]
\end{lemma}

\begin{proof}
Under the assumption of a unique optimizer $\gamma_t^\star$ in (\ref{eq: cost}), we get 
\[
\ce(x,y)= \int_0^1 \|\dot\gamma_t^\star\|^2 \dd t.
\]
To obtain the derivative, first pick $x' \in \Ne$ sufficiently close to $x$, such that there exists a unique geodesic connecting $x$ and $x'$ in $\Ne$. We will establish the limit as $x'\to x$. Then, due to the continuous dependence of $\gamma^\star_t$ on $x$, we can ``slightly modify'' $\gamma^\star_t$ to become $\gamma_t'$, which now connects $x'$ and $y$, instead of $x$ and $y$, and thus 
\begin{equation*}
    \ce\rnd{x', y} \leq \int_0^1 \|{\dot{\gamma}_t}'\|^2 \dd t,
\end{equation*}
as the path $\gamma_t'$ might not be optimal.
The formula of first variation then gives 
\[
\ce(x',y) \le \int_0^1 \left(\|\dot\gamma_t^\star\|^2 - 2 \langle \dot\gamma_0^\star, x'-x \rangle\right) \mathrm{d}t + \mathrm{o}\left( \sup_{0\le t\le 1} \|\gamma_t^\star - \gamma_t'\|\right), 
\]
which quantifies the difference between the action of the optimal curve connecting $x'$ and $y$, and the action of the slightly perturbed curve $\gamma_t'$. \\

A suitable choice of $\gamma'_t$ ensures that $\sup_{0\le t\le 1} \|\gamma_t^\star - \gamma'_t\|= \|x'-x\|$. 
Therefore, we obtain by subtracting the equality $\ce(x,y)= \int_0^1 \|\dot\gamma_t^\star\|^2 \dd t$ from the above inequality
\[
\ce(x',y) -\ce(x,y)  \le - 2 \langle \dot\gamma_0^\star, x'-x \rangle + \mathrm{o}\left(\|x'-x\| \right), 
\]
which establishes that  $- 2  \dot\gamma_0^\star$ is an upper gradient of $\ce(x,y)$. \\

By exchanging the roles of $x$ and $x'$, coupled with the continuity of the geodesics, a similar argument then yields a matching lower bound, and the claim follows.
\end{proof}

\subsection{Existence of an optimal map}

To prove the existence of a unique coupling $\pi^\star_\eps$ in (\ref{eq: OTeps}), one possibility is to prove the existence of a unique map $T$ such that 
each point $x \in \Ne$ is sent to a unique $y=T(x) \in \Ne$. Given our previous discussion on the cost function, the standard proofs for uniqueness of a coupling  do not apply in this setting.
Indeed, the latter proofs would usually try to extract such a map $T$ from the dual formulation \eqref{eq: dual}, i.e., from the relationship
\[
\phi(x) + \phi^c( y) = \ce(x,y) \quad \text{ for } (x,y) \in \operatorname{spt}(\pi_\eps^\star).
\]
The idea of deriving $y=T(x)$ is to argue along the lines of the implicit function theorem. 
Assuming differentiability of $\phi$, the above equality gives
\[
\nabla\phi(x) = \nabla_x \ce(x,y), 
\]
from which the transport plan could be obtained if for every fixed $x \in \Ne$, the map $y \mapsto \nabla_x \ce(x, y)$ was injective. However, contrary to the case covered in \cite[Remark~10.32]{villani2008optimal}, the minimizing curves $\gamma_t^\star$ are not uniquely determined by their starting point $x$ and their initial velocity $\res{\dot{\gamma}_t^\star}{t = 0}$, hence the set of $y$ corresponding to a pair $\rnd{x, \res{\dot{\gamma}_t^\star}{t = 0}}$ is generally not a singleton. \\

The issue of absence of injectivity also appears in the case of \textit{branching geodesics}, which has attracted some attention, see for instance \cite{cavalletti2012optimal} and \cite{erbar2021gradient}. This means that for two geodesics $\gamma_t^\star$ and $\beta_t^\star$ with the same starting point $x$ and initial velocities $\res{\dot{\gamma}_t^\star}{t = 0} = \res{\dot{\beta}_t^\star}{t = 0}$, there exists a point in time $t_0 \in \rnd{0, 1}$, such that 
\begin{equation*}
    \forall t \in \sqr{0, t_0}: \, \gamma_t^\star = \beta_t^\star \quad \text{and} \quad \forall t \in \left(t_0, 1\right]: \, \gamma_t^\star \neq \beta_t^\star.
\end{equation*}
In the network setting this phenomenon arises as soon as there exists at least one junction (of at least three pipes), such as the example in following graph:

\begin{center} 
    \begin{tikzpicture}[scale=2, every node/.style={font=\small}]
        \coordinate (x) at (-1,0); \coordinate (junction) at (0,0);
        \coordinate (gamma) at (0.71,0.71);
        \coordinate (beta) at (0.71,-0.71);
        \draw[gray!60, dashed, line width=0.5pt] ([yshift=5pt]x)--([xshift=-3.54pt,yshift=5pt]junction);
        \draw[line width=0.9pt, babyblue] ([yshift=0.22pt]x)--([yshift=0.22pt]junction); \draw[line width=0.9pt, red] ([yshift=-0.22pt]x)--([yshift=-0.22pt]junction);
        \draw[gray!60, dashed, line width=0.5pt] ([yshift=-5pt]x)--([xshift=-2.5pt,yshift=-5pt]junction);
        \draw[gray!60, dashed, line width=0.5pt] ([xshift=-3pt,yshift=4.7pt]junction)--([xshift=-3.54pt,yshift=3.54pt]gamma);
        \draw[line width=1.6pt, babyblue] (junction)--(gamma);
        \draw[gray!60, dashed, line width=0.5pt] ([xshift=7.07pt]junction)--([xshift=3.54pt,yshift=-3.54pt]gamma);
        \draw[gray!60, dashed, line width=0.5pt] ([xshift=-3pt,yshift=-4.7pt]junction)--([xshift=-3.54pt,yshift=-3.54pt]beta);
        \draw[line width=1.6pt, red] (junction)--(beta);
        \draw[gray!60, dashed, line width=0.5pt] ([xshift=7.07pt]junction)--([xshift=3.54pt,yshift=3.54pt]beta);
        \filldraw[black] (x) circle (1pt) node[left] {$x$};
        \filldraw[babyblue] (gamma) circle (1pt) node[right] {$y^{\gamma}$};
        \filldraw[red] (beta) circle (1pt) node[right] {$y^{\beta}$};
    \end{tikzpicture}
\end{center}


\subsubsection{Main results}

We start by establishing that the optimal transport cost (\ref{eq: OTeps}) converges for $\eps \to 0$. 
To this end, let us first recall that a measure $\mu$ is called atomeless if for all $\mu$-measurable sets $A \subset \R^3$ with $\mu\rnd{A} > 0$, there exists a $\mu$-measurable set $B \subset A$ with $0 < \mu\rnd{B} < \mu\rnd{A}$.

\begin{proposition}
    Assume that the sequences $(\mue)$ and $(\nue)$ are sequences of absolutely continuous probability measures with respect to the Lebesgue measure,  that further converge weakly to atomeless probability measures $\mu_0, \nu_0$ supported on $\No$.  Then, 
    \[
    \operatorname{OT}(\mue, \nue, \ce) \to \operatorname{OT}(\muo, \nuo, \co), \quad \text{ as } \eps \to 0.
    \]
\end{proposition}
\begin{proof}
Denote by $P$ the projection of a point from $\Ne$ to the closest point in $\No$, with respect to the Euclidean distance, and in the case of the closest point in $\No$ not being unique (which can happen at the junction of pipes), pick an arbitrary point from the contenders in $\No$. \\

For Lebesgue-almost every pair $(x, y) \in \Ne \times \Ne$, it holds
$$\ce(x,y) \le \co(P(x),P(y)) + 2\eps^2, $$ as this construction amounts to picking a non-optimal path between $x$ and $y$. Recall that for any $\mu$-nullset $E$, integrating a function $f$ with $f(x)= +\infty$ for all $x\in E$ evaluates at 0, so $\int_E f(x) \dd\mu(x)=0$. Therefore, we do not need the previous inequality for $x$ or $y \in \R^3 \setminus \Ne$.
This then yields
\begin{align*}
    \int_{\R^3 \times \R^3} \ce\rnd{x, y} \dd \pie\rnd{x, y} & \le \int_{\R^3 \times \R^3} c_0(P(x),P(y)) \dd \pie(x,y) + 2 \eps^2 \\
    & = \int_{\R^3 \times \R^3} c_0 \dd (P, P)_\#\pie + 2 \eps^2,
\end{align*}
where for any probability measure $\mu \in \Prb{X}$, the notation $F_\# \mu$ for a measurable function $F: X \rightarrow Y$ corresponds to the pushforward measure defined as $F_\# \mu \in \Prb{Y}$ with $F_\# \mu\rnd{B} \coloneqq \mu\rnd{F^{-1}\rnd{B}}$ for measurable $B \subseteq Y$. \\

Now, for any coupling $\pie$ of the marginals $\mue, \nue $, the term $(P,P)_\#\pi_\eps$ corresponds to a coupling between $P_\# \mue \in \Prb{\No}$ and  $P_\# \nue \in  \Prb{\No}$.  Optimizing over $\pie$ is thus equivalent to optimizing over couplings of $P_\# \mue $ and  $P_\# \nue $ and so, we obtain
\begin{equation*}
    \inf_{\pie \in \Pi(\mu_\eps, \nu_\eps)} \crl{\int_{\R^3 \times \R^3} c_0 \dd (P, P)_\#\pie} = \inf_{\pi_\eps^P \in \Pi(P_\# \mu_\eps, P_\# \nu_\eps)} \crl{\int_{\R^3 \times \R^3} c_0 \dd \pi_\eps^P}.
\end{equation*}

Further, $P_\# \mue, P_\# \nue $
converge weakly to $P_\# \muo = \muo$ and $P_\# \nuo = \nuo$, by the continuous mapping theorem. Note that the assumption on the absence of atoms of the limits $\muo, \nuo$ is required here, because of the absence of continuity of the projection $P$ at junctions, which are not charged as $P$ is continuous $\muo$- and $\nuo$-almost everywhere. \\

One can then invoke Theorem~1.51 in \cite{Santambrogio2015} (where we once again use the argument that $P_\#\mue (\R^3 \setminus \No)=0 $ and $P_\#\nue (\R^3 \setminus \No)=0 $ for all $\eps>0$, as well as $P_\#\muo (\R^3 \setminus \No)=0 $ and $P_\#\nuo (\R^3 \setminus \No)=0$ to circumvent the non-compactness of $\R^3$ and $\co$ being non-continuous) to conclude that 
\[
\inf_{\pi_\eps^P \in \Pi(P_\# \mu_\eps, P_\# \nu_\eps)} \crl{\int_{\R^3 \times \R^3} c_0 \dd \pi_\eps^P} \to
\inf_{\pio \in \Jprb{\muo, \nuo}} \crl{\int_{\R^3 \times \R^3} c_0 \dd \pio} = \operatorname{OT}\rnd{\muo, \nuo, \co},
\]
as $\eps \to 0$, which yields
\begin{equation*}
    \operatorname{OT}\rnd{\mue, \nue, \ce} \leq \operatorname{OT}\rnd{\muo, \nuo, \co}.
\end{equation*}

In the other direction, for Lebesgue-almost every pair $(x, y) \in \Ne \times \Ne$, we have $$\ce(x,y) \geq \co(P(x),P(y)) - K\eps^2,$$ where $K>0$ is some constant depending on the network. This constant $K$ is necessary as the network $\Ne$ allows for shorter paths around a corner (at junctions) compared to the graph $\No$. Since the number of pipes and junctions is finite, the amount of path length that can be ``saved'' in $\Ne$, compared to $\No$, can be bounded from above. \\

The length of the path obtained by projecting first two points $x, y \in \Ne$ on $\No$, and then joining them by the shortest path on $\No$ must thus stay within a term $K\eps^2$. The claim then follows from a similar argument as above. We have
\begin{align*}
    \int_{\R^3 \times \R^3} \ce\rnd{x, y} \dd \pie\rnd{x, y} & \geq \int_{\R^3 \times \R^3} c_0(P(x),P(y)) \dd \pie(x,y) - K \eps^2 \\
    & = \int_{\R^3 \times \R^3} c_0 \dd (P, P)_\#\pie - K \eps^2
\end{align*}
with the equality from before
\begin{equation*}
    \inf_{\pie \in \Pi(\mu_\eps, \nu_\eps)} \crl{\int_{\R^3 \times \R^3} c_0 \dd (P, P)_\#\pie} = \inf_{\pi_\eps^P \in \Pi(P_\# \mu_\eps, P_\# \nu_\eps)} \crl{\int_{\R^3 \times \R^3} c_0 \dd \pi_\eps^P},
\end{equation*}
where the convergence of the right-hand-side term for $\eps \to 0$ gives us the desired inequality
\begin{equation*}
    \operatorname{OT}\rnd{\mue, \nue, \ce} \geq \operatorname{OT}\rnd{\muo, \nuo, \co}.
\end{equation*}
Therefore, for $\eps \to 0$, based on the sandwich theorem, we can conclude
\begin{equation*}
    \operatorname{OT}(\mue, \nue, \ce) \to \operatorname{OT}(\muo, \nuo, \co).\qedhere
\end{equation*}
\end{proof}

The question of establishing the convergence of $\pi_\eps$ as $\eps \to 0$ is thus a stability issue with respect to both measures $\mue, \nue$ and cost $\ce$. This question is addressed in Chapter~5 of \cite{villani2008optimal}, for instance. This being said, the results therein do not apply directly to our context because our cost function does not converge uniformly on $\Ne$. \\

Our main theorem is the following.

\begin{theorem}
Let $\rnd{\mu_\eps}_\eps, \rnd{\nu_\eps}_\eps$ be sequences of probability measures with support $\mathcal{N}_\eps$, and consider the sequence of cost functions $\rnd{\ce}_\eps$ defined in \eqref{eq: cost}.
Moreover, by $\pi_\eps^\star$ we denote any optimizer of problem \eqref{eq: OTeps}, which might not be uniquely determined.
If $\pi_0^\star$, the solution of \eqref{eq: OTeps} for $\eps =0$, is unique, then (all)
\[
    \pi_\eps^\star \to \pi_0^\star
\]
weakly as $\eps\to 0$.
Otherwise, each convergent subsequence of $\pi_\eps^\star$ converges to one of the optimizers $\pio^\star$ of \eqref{eq: OTeps}.
\end{theorem}

\begin{proof}
As $\mu_\eps$ and $\nu_\eps$ are weakly convergent sequences, by
Prokhorov’s theorem they constitute tight sets, so for all $\xi > 0$ there exist compact sets $K_\mu, K_\nu \subset \R^3$ such that
\begin{equation*}
    \sup_{0 < \eps < 1} \crl{\mue\rnd{\R^3 \backslash K_\mu}}, \sup_{0 < \eps < 1} \crl{\nue\rnd{\R^3 \backslash K_\nu}} < \xi.
\end{equation*}
Then, by Lemma 4.4 in \cite{villani2008optimal}, the couplings $\pi_\eps \in \Jprb{\mue, \nue}$ all lie in a tight set of $\R^3 \times \R^3$. \\

First, let us recall that optimizers $\pi_\eps^\star$ of (\ref{eq: OTeps}) exist, i.e., the infimum is attained as the set of couplings is compact and $\pi\mapsto \int c \dd\pi$ is a lower-bounded, linear functional. Due to the tightness of the couplings, applying Prokhorov’s theorem to a sequence of optimizers $\pi_\eps^\star$ , we can extract a subsequence, still denoted $\pi_\eps^\star$ for simplicity, which converges weakly to some $\pi_0^* \in \Pi(\mu_0, \nu_0)$. We aim at proving that this limiting coupling $\pi_0^*$ is an optimizer of (\ref{eq: OTeps}) for $\eps = 0$, by showing that it is $c_0$-cyclically monotone. \\

Note that in $\Ne\times\Ne$, $c_{\!\eps}$ is continuous and upper bounded, implying that  $\operatorname{OT}\rnd{\mue, \nue, \ce}$ is bounded as well. We can thus apply Theorem~\ref{thm: KantDual} and as $\pi_\eps^\star$ is optimal, each $\pi_\eps^\star$ is concentrated on a $\ce$-cyclically monotone set. Moreover, for any $N \geq 1$ pick arbitrary points $(x_1, y_1), \ldots, (x_N, y_N)$ from the support of $\pi_\eps^\star$ and denote this set of points by $\pi_\eps^{\otimes N}$. Additionally, define the set $C_{\!\eps}(N)$ as the set of all sets of $N$ points $\crl{(\tilde x_1,\tilde y_1), \ldots,(\tilde x_N,\tilde y_N)}$ from $\Ne\times \Ne$ such that 
\begin{equation}
\label{eq: cyclMono}
    \sum_{i=1}^N c_{\!\eps}(\tilde x_i, \tilde y_i) \le  \sum_{i=1}^N c_{\!\eps}(\tilde x_i, \tilde y_{i+1}), 
\end{equation}
with the convention $\tilde y_{N+1}=\tilde y_1$. \\

By the $\ce$-cyclical monotonicity of $\pi_\eps^\star$, the set $\pi_\eps^{\otimes N}$ is thus concentrated on $C_{\!\eps}(N)$.
Further, as the  cost function $c_{\!\eps}$ is continuous on $\mathcal{N}_\eps\times \mathcal{N}_\eps$,
the set of sets $C_{\!\eps}(N)$ is closed. As a next step, we show that $\pi_\eps^{\otimes N}$ is also concentrated on another set. Let $\delta > 0$ and $N$ be given, define
$C_\delta(N)$ as the set of all sets of $N$ points $\crl{(\hat{x}_1,\hat y_1), \ldots, ( \hat x_N,\hat y_N)}$ from $\Ne \times \Ne$ such that
\begin{equation}
\label{eq: cyclMono2}
    \sum_{i=1}^N c_0(\hat x_i,\hat y_i) \le  \sum_{i=1}^N c_0(\hat x_i, \hat y_{i+1}) + \delta,
\end{equation}
which corresponds to a $\delta$-approximate $c_0$-cyclical monotonicity. We now want to prove that, for $\eps$ small enough, the set $\pi_\eps^{\otimes N}$ is concentrated on $C_\delta(N)$. We do so by distinguishing two cases: First, either at least one of the $x_i$'s or the $y_i$'s from $\pi_\eps^{\otimes N}$ is not in $\mathcal{N}_0$, in which case both sides of the inequality \eqref{eq: cyclMono2} are $+\infty$ and the inequality holds trivially. Second, if $\pi_\eps^{\otimes N} \subset \No \times \No$ observe that for all $x, y \in \No$ we get
\[
\ce(x,y) \le  \co(x,y),
\]
which together with the $\ce$-cyclical monotonicity of $\pi_\eps^\star$, so \eqref{eq: cyclMono}, results in
\begin{equation*}
    \sum_{i=1}^N \ce(x_i, y_i) \le  \sum_{i=1}^N \ce(x_i, y_{i+1}) \le  \sum_{i=1}^N \co(x_i, y_{i+1}).
\end{equation*}
Finally, $c_{\!\eps}(x,y)$ converges uniformly to $c_0(x,y)$ for $x,y \in \mathcal{N}_0$, so there exists $\eps>0$ small enough such that for all $x, y \in \No$ we obtain
\[
\lvert c_{\!\eps}(x,y) -c_0(x,y)\vert \le \frac{\delta}{N}.
\]
Therefore, 
\eqref{eq: cyclMono2} indeed also holds for all $x_i, y_i$'s in $\No$, which yields $\operatorname{spt}\rnd{\pi_\eps^{\otimes N}} \subseteq C_\delta\rnd{N}$ for $\eps$ small enough. Since $C_\delta\rnd{N}$ is a closed set, the weak convergence of $\pi_\eps^\star$ to $\pi_0^*$ also implies $\operatorname{spt}\big(\pi_0^{\otimes N}\big) \subseteq C_\delta\rnd{N}$. To show that $C_\delta\rnd{N}$ is closed, a sublevel set argument for the function $f: \rnd{\R^3 \times \R^3}^N \to \R \cup \crl{+ \infty}, \quad \big(\rnd{x_1, y_1}, \dots, \rnd{x_N, y_N}\big) \mapsto \sum_{i=1}^N c_0\rnd{x_i, y_i} - \sum_{i=1}^N c_0\rnd{x_i, y_{i+1}}$ with $C_\delta\rnd{N} = \crlm{z}{f\rnd{z} \leq \delta}$ can be used. \\

As the inequality \eqref{eq: cyclMono2} holds for all $\big(\rnd{x_1, y_1}, \dots, \rnd{x_N, y_N}\big) \in \operatorname{spt}\rnd{\pi_0^{\otimes N}}$ for all $\delta > 0$, in the limit $\delta \to 0$, we get
\begin{equation*}
    \sum_{i=1}^N c_0\rnd{x_i, y_i} \leq \sum_{i=1}^N c_0\rnd{x_i, y_{i+1}}
\end{equation*}
for all $\big(\rnd{x_1, y_1}, \dots, \rnd{x_N, y_N}\big) \in \operatorname{spt}\rnd{\pi_0^{\otimes N}}$. Therefore, $\pi_0^*$ is $c_0$-cyclically monotone, which guarantees the optimality of $\pi_0^*$ in \eqref{eq: OTeps} for $\eps = 0$ by Theorem~\ref{thm: KantDual}. \\

Finally, in the case of a unique $\pi_0^\star$, we conclude the convergence of each subsequence $\pi_\eps^\star$ to $\pi_0^\star$, so that the whole sequence converges to $\pi_0^\star$ by a classical subsequence argument.


\end{proof}

\begin{remark}
In the somewhat trivial case of a network consisting of a single pipe, all the presented results still apply. The unique difference is that then the initial velocity is sufficient to characterize a geodesic, and uniqueness of the transport plan holds for each $\eps>0$ under the assumption of the measures $\mue, \nue$ having a density with respect to the Lebesgue measure.
\end{remark}

\subsection{Towards a dynamical understanding}
\label{sec: DynUnderst}
The issue underlying the absence of uniqueness of optimizers $\pi_\eps^\star$ was pointed out above and relates to the fact that minimizing curves in the network change directions. Still, one intuits that minimal additional knowledge could suffice to still get a transport plan in ``well-behaved'' examples. To make a step in that direction and explain the open questions, let us consider the space-time lifting of the problem, i.e., 
we extend the measures $\mue$ and $\nue$ to $\mu_\eps \otimes \delta_{t=0}$ and $\nu_\eps \otimes \delta_{t=1}$ and utilize the cost function 
\[
    \tilde{c}\big( (x,s), (y,t) \big) = \inf \left\{ \int_s^t \|\dot \gamma_u\|^2 + \iota_{\mathcal{N}_\eps}(\gamma_u) \dd u :    \gamma(s)=x, \gamma(t)=y \right\},
\]
for a piecewise absolutely continuous curve $\gamma_u$ and $s, t \in \sqr{0, 1}$. \\

From the characterization of the dual problem, if the following quantities were well-defined and the function
\[
G_{(x, s)}(y,t):=(y,t) \mapsto 
\begin{pmatrix} \nabla_x \tilde{c}\big((x,s), (y,t)\big)\\
\partial_s \tilde{c}\big((x,s), (y,t)\big)
\end{pmatrix}
\]
was invertible, then the optimal transport map in the space-time domain
is given by 
\[
T(x,t) = G_{(x, s)}^{-1}
\begin{pmatrix} \nabla_x \psi(x,s)\\
\partial_s\psi(x,s)
\end{pmatrix},
\]
where $\psi(x,s)$ is an optimal dual potential.

\subsection{Stability with respect to the gas network topology}

An important question is to know how the optimal gas transport depends on the topology of the network, and in particular how the cost of transport would be affected in case of modifications of the network structure. \\

To better understand this, let us introduce the following setting. We consider two probability measures $\mu,\nu$ that are supported on $\check{\mathcal{N}}_0:=\mathcal{N}_0\setminus\mathcal{M}$, where $\mathcal{M}$ is a set of edges in the central portion of the graph, which are removed to get a new graph structure. Usually, one thinks of $\mu$ as a source and $\nu$ as a target, such that $\mu$ is flowing to $\nu$ through $\mathcal{M}$ (at least partially). As a next step we can add new edges as a substitute for $\mathcal{M}$, which we denote by $\widetilde{\mathcal{M}}$ to get the new graph $\widetilde{\mathcal{N}}:=\mathcal{N}_0\setminus\mathcal{M} \cup \widetilde{\mathcal{M}}$. Note that we require $\widetilde{\mathcal{N}}$ to be still connected to allow for a feasible transport. Then, we define the modified cost function
\[
   \tilde c(x,y) := \inf_{\gamma \in \mathcal{C}(x,y)} \int_0^1 \|\dot \gamma_t\|^2 + \iota_{\widetilde{\mathcal{N}}}(\gamma_t) \dd t,
\]
as well as the modified optimal transport problem 
\[
\operatorname{OT}(\mu, \nu, \tilde{c}) :=\inf_{\pi \in \Pi(\mu, \nu)} \int_{\check{\mathcal{N}}_0\times \check{\mathcal{N}}_0} \tilde{c} \dd \pi.
\]
Recall that $\mu$ and $\nu$ are supported on $\check{\mathcal{N}}_0$ and thus $\pi$ is supported on $\check{\mathcal{N}}_0\times \check{\mathcal{N}}_0$, however the transport through the new section of the graph $\widetilde{\mathcal{M}}$ is accounted for in the cost function $\tilde{c}$ as the feasible shortest-path curves $\gamma$ live on the space domain $\widetilde{\mathcal{N}}$. \\

We denote by $O^\star$ and $\widetilde{O}^\star$ the set of optimal couplings in both problems, i.e., 
\[
O^\star :=\left\{ \pi \in \Pi(\mu, \nu) : \int_{\check{\mathcal{N}}_0\times \check{\mathcal{N}}_0} c \dd \pi = \inf_{\pi \in \Pi(\mu, \nu)} \int_{\check{\mathcal{N}}_0\times \check{\mathcal{N}}_0} {c} \dd \pi\right \},
\]
\[
\widetilde{O}^\star :=\left\{ \tilde{\pi} \in \Pi(\mu, \nu) : \int_{\check{\mathcal{N}}_0\times \check{\mathcal{N}}_0} \tilde{c} \dd\tilde{\pi}= \inf_{\tilde{\pi} \in \Pi(\mu, \nu)} \int_{\check{\mathcal{N}}_0\times \check{\mathcal{N}}_0} \tilde{c} \dd \tilde{\pi}\right \}.
\]
The following theorem relates the values of the optimal costs between the two situations, i.e., before and after modification of the gas network.

\begin{theorem}
\label{thm: Stability}
In the above setting, it holds true that
\begin{align*}
\big \vert\operatorname{OT}(\mu, \nu, \tilde{c})  - \operatorname{OT}(\mu, \nu, {c}) \big \vert
&\le 
\max\crl{
\inf_{\pi \in O^\star} \int_{\check{\mathcal{N}}_0\times \check{\mathcal{N}}_0} \vert \tilde{c}- c \vert \mathrm{d}\pi, 
\inf_{\tilde{\pi} \in \widetilde{O}^\star} \int_{\check{\mathcal{N}}_0\times \check{\mathcal{N}}_0}
\vert \tilde{c}- c \vert \mathrm{d}\tilde{\pi}}\\
&\le 
\|  \tilde{c}- c\|_\infty.
\end{align*}
\end{theorem}

\begin{remark}
The middle bound expresses the natural intuition that one might have: The difference in cost between the two network topologies depends on the difference between the costs weighted by the mass transported between these points. 
\end{remark}
\begin{remark}[Deleting one connection]
If the change consists in removing a single pipe within the network, then for all $x,y \in \check{\mathcal{N}}_0$ it holds true that
\[
c(x,y) \le \tilde {c}(x,y),
\]
as each path that was going through this pipe now should be at least as long. Thus, 
\[
\operatorname{OT}(\mu, \nu, \tilde{c})  -  \operatorname{OT}(\mu, \nu, {c}) \le \inf_{\pi \in O^\star} \int_{\check{\mathcal{N}}_0\times \check{\mathcal{N}}_0} (\tilde{c}-c)\ \mathrm{d}\pi^\star.
\]
\end{remark}
\begin{proof}[Proof of Theorem~\ref{thm: Stability}] Plugging in the definitions of OT, we get the equality
\[
\operatorname{OT}(\mu, \nu, \tilde{c})  - \operatorname{OT}(\mu, \nu, {c})  = \inf_{\pi \in \Pi(\mu, \nu)} \int_{\check{\mathcal{N}}_0\times \check{\mathcal{N}}_0} \tilde{c} \dd \pi -  \inf_{\pi \in \Pi(\mu, \nu)} \int_{\check{\mathcal{N}}_0\times \check{\mathcal{N}}_0} c \dd \pi,
\]
and then taking any $\pi^\star$ that is an optimal coupling for the problem involving the cost $c$, we obtain 
\begin{align*}
\operatorname{OT}(\mu, \nu, \tilde{c})  - \operatorname{OT}(\mu, \nu, {c}) & \leq \int_{\check{\mathcal{N}}_0\times \check{\mathcal{N}}_0} \tilde{c} \dd \pi^\star -  \underbrace{\inf_{\pi \in \Pi(\mu, \nu)} \int_{\check{\mathcal{N}}_0\times \check{\mathcal{N}}_0} c \dd \pi}_{= \int_{\check{\mathcal{N}}_0\times \check{\mathcal{N}}_0} c \dd \pi^\star} \\
& = \int_{\check{\mathcal{N}}_0\times \check{\mathcal{N}}_0} (\tilde{c}-c)\ \mathrm{d}\pi^\star \\
& \le  \int_{\check{\mathcal{N}}_0\times \check{\mathcal{N}}_0} \vert\tilde{c}-c\vert\ \mathrm{d}\pi^\star.
\end{align*}
So that, because of the arbitrariness of $\pi^\star\in \mathcal{O}^\star$, 
\[
\operatorname{OT}(\mu, \nu, \tilde{c})  - \operatorname{OT}(\mu, \nu, {c})  \le  \inf_{\pi \in \mathcal{O}^\star}\int_{\check{\mathcal{N}}_0\times \check{\mathcal{N}}_0} \vert\tilde{c}-c\vert\ \mathrm{d}\pi.
\]
For the other direction, apply the same idea, picking a coupling $\tilde{\pi}^\star \in \widetilde{\mathcal{O}}^\star$, to get 
\[
\operatorname{OT}(\mu, \nu, c )  - \operatorname{OT}(\mu, \nu, \tilde{c})  \le  \inf_{\tilde{\pi} \in \widetilde{\mathcal{O}}^\star}\int_{\check{\mathcal{N}}_0\times \check{\mathcal{N}}_0} \vert\tilde{c}-c\vert\ \mathrm{d}\tilde{\pi}.
\]
The first part of the claim follows. For the second part, note that $\|c -\tilde{c}\|_\infty \geq \int |c-\tilde{c}| \dd \pi$, as $\int  \dd \pi
=1$.\end{proof}

\section{Numerical illustrations}

In this final section, we exhibit (estimated) trajectories of a shortest path metric constrained to stay within the network. Based on a uniform grid in the network, we construct a graph linking each grid point (or pixel for the later representation) to its 8 surrounding nearest neighbors and itself with the following squared distance matrix 
\[
\begin{pmatrix}
2 &1 &2 \\1 & 0 &1\\ 2& 1 &2
\end{pmatrix}.
\]
Then, given random points on the structure, we compute the cost matrix based on the shortest path between each pair of points (obtained using Dijkstra's algorithm) and solve a linear assignment problem. \\

We finally represent the trajectories between the matched points in Figure~\ref{fig:MoveInThePipe} below. We clearly see the fact that in this example geodesics branch and also collapse. This explains why the standard proofs of uniqueness of transport plans do not apply \textemdash recall Section~\ref{sec: PropCost}  \textemdash and why the construction in Section~\ref{sec: DynUnderst} would be required to derive a potential existence of a unique transport plan. 

\begin{figure}[h!]
    \centering
    \includegraphics[width=1\linewidth]{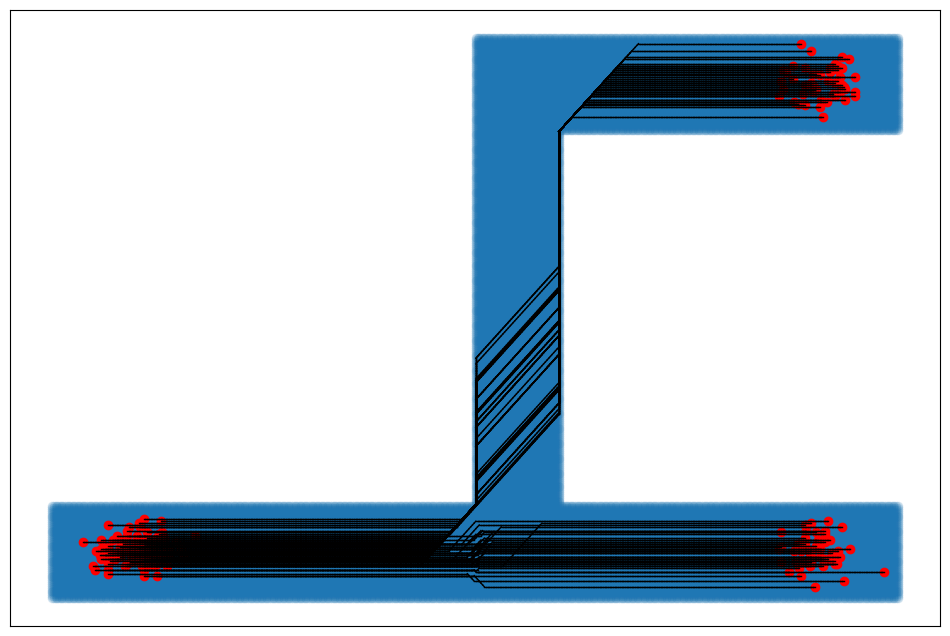}
    \caption{\label{fig:MoveInThePipe}Example of estimated trajectories based on a space-discretized version of the problem. The source measure is the set of points in the lower left corner, while the target measure are the two accumulations of points at the right hand side.}
\end{figure}

\newpage

%% file: references.tex
%
%

%% file: main.bbl
\begin{thebibliography}{99.}%
%

\bibitem{ambrosio2005gradient}
Ambrosio, L., Gigli, N., Savar{\'e}, G.: Gradient Flows: In Metric Spaces and in the Space of Probability Measures. Springer (2005)

\bibitem{Arjovsky2017}
Arjovsky, M., Chintala, S., Bottou, L.: Wasserstein Generative Adversarial Networks. Proceedings of the 34th International Conference on Machine Learning (ICML) (2017) doi: 10.5555/3305381.3305404

\bibitem{benamou2000computational}
Benamou, J.D., Brenier, Y.: A computational fluid mechanics solution to the Monge–Kantorovich mass transfer problem. Numerische Mathematik, \textbf{84}(3), 375--393 (2000) doi:10.1007/s002110050002

\bibitem{burger2023dynamic}
Burger, M., Humpert, I., Pietschmann, J.F.: Dynamic optimal transport on networks. ESAIM: Control, Optimisation and Calculus of Variations, \textbf{29} (2023) doi: 10.48550/arXiv.2101.03415

\bibitem{CARDALIAGUET201143}
Cardaliaguet, P., Jimenez, C.: Optimal transport with convex obstacle. Journal of Mathematical Analysis and Applications, \textbf{381}(1), 43--63 (2011) doi: 10.1016/j.jmaa.2011.04.007

\bibitem{CarioniKrautzPietschmann2025_PreferentialPaths}
Carioni, M., Krautz, J., Pietschmann, J.F.: Dynamic Optimal Transport with Optimal Preferential Paths. arXiv preprint arXiv:2504.03285, (2025) doi: 10.48550/arXiv.2504.03285

\bibitem{CarioniKrautzPietschmann2025_StarShapedGraphs}
Carioni, M., Krautz, J., Pietschmann, J.F.: Dynamic Optimal Transport with Optimal Star Shaped Graphs. arXiv preprint arXiv:2506.15007, (2025) doi: 10.48550/arXiv.2506.15007

\bibitem{cavalletti2012optimal}
Cavalletti, F.: Optimal transport with branching distance costs and the obstacle problem. SIAM Journal on Mathematical Analysis, \textbf{44}(1), 454--482 (2012) doi: 10.1137/100801433

\bibitem{chizat2020lecture}
Chizat, L.: Lecture 3: Wasserstein Space. (2020) \url{https://lchizat.github.io/files2020ot/lecture3.pdf}

\bibitem{Cuturi2013}
Cuturi, M.: Sinkhorn Distances: Lightspeed Computation of Optimal Transport. Advances in Neural Information Processing Systems (NeurIPS), \textbf{26} (2013) doi:10.48550/arXiv.1306.0895

\bibitem{Erb14}
Erbar, M.: Gradient flows of the entropy for jump processes. Annales de l'Institut Henri Poincar{\'{e}}, Probabilit{\'{e}}s et Statistiques, \textbf{50}(3), 920 -- 945 (2014) doi: 10.1214/12-AIHP537

\bibitem{erbar2021gradient}
Erbar, M., Forkert, D., Maas, J., Mugnolo, D.: Gradient flow formulation of diffusion equations in the Wasserstein space over a metric graph. arXiv preprint arXiv:2105.05677 (2021) doi: 10.48550/arXiv.2105.05677

\bibitem{Fazeny_2025}
Fazeny, A., Burger, M., Pietschmann, J.F.: Optimal transport on gas networks. European Journal of Applied Mathematics, 1--33, Apr. (2025) doi: 10.1017/S0956792525000051

\bibitem{Frogner2015}
Frogner, C., Zhang, C., Mobahi, H., Araya, M., Poggio, T.: Learning with a Wasserstein Loss. Advances in Neural Information Processing Systems (NeurIPS), \textbf{28} (2015) doi: 10.48550/arXiv.1506.05439

\bibitem{JordanKinderlehrerOtto1997}
Jordan, R., Kinderlehrer, D., Otto, F.: Free Energy and the Fokker–Planck Equation. Physica D: Nonlinear Phenomena, \textbf{107}(2–4), 265--271 (1997) doi:  10.1016/S0167-2789(97)00093-6

\bibitem{JordanKinderlehrerOtto1998}
Jordan, R., Kinderlehrer, D., Otto, F.: The Variational Formulation of the Fokker–Planck Equation. SIAM Journal on Mathematical Analysis, \textbf{29}(1), 1--17 (1998) doi: 10.1137/S0036141096303359

\bibitem{Hallin2020MultivariateGT}
Hallin, M., Mordant, G., Segers, J.: Multivariate goodness-of-fit tests based on Wasserstein distance. Electronic Journal of Statistics (2020) doi: 10.1214/21-EJS1816


\bibitem{Heinze2024}
Heinze, G., Pietschmann, J.F., Schlichting, A.: Gradient flows on metric graphs with reservoirs: Microscopic derivation and multiscale limits. arXiv preprint \href{https://arxiv.org/abs/2412.16775}{arXiv:2412.16775} (2024) doi: 10.48550/arXiv.2412.16775

\bibitem{Kantorovich1942}
Kantorovich, L.V.: On the translocation of masses. Management Science, \textbf{5}(1), 1–-4 (1958) doi: 10.1007/s10958-006-0049-2

\bibitem{Kolouri2017}
Kolouri, S., Park,  S.R., Thorpe, M., Slepčev, D., Rohde, G.K.: Optimal Mass Transport: Signal Processing and Machine-Learning Applications. IEEE Signal Processing Magazine, \textbf{34}(4), 43-–59 (2017) doi: 10.1109/MSP.2017.2695801

\bibitem{lavenant2021towards}
Lavenant, H., Zhang, S., Kim, Y.H., Schiebinger, G.: Towards a mathematical theory of trajectory inference. arXiv preprint arXiv:2102.09204 (2021) doi: 10.1214/23-AAP1969

\bibitem{LieroMielkePeletierRenger2017}
Liero, M., Mielke, A., Peletier, M.A., Renger, D.R.M.: On microscopic origins of generalized gradient structures. Discrete and Continuous Dynamical Systems-Series S, \textbf{10}(1):1 (2017) doi:10.3934/dcdss.2017001

\bibitem{lisini2007characterization}
Lisini, S.: Characterization of absolutely continuous curves in Wasserstein spaces. Calculus of variations and partial differential equations, \textbf{28}(1), 85--120 (2007) doi: 10.1007/s00526-006-0032-2

\bibitem{Mielke2023}
Mielke, A.: An Introduction to the Analysis of Gradient Systems. Lecture Notes, WIAS Preprint 3022, Berlin (2023) 10.48550/arXiv.2306.05026

\bibitem{MielkePeletierRenger2014}
Mielke, A., Renger, D.R.M., Peletier, M.A.: On the Relation Between Gradient Flows and the Large–Deviation Principle, with Applications to Markov Chains and Diffusion. Potential Analysis, \textbf{41}(4), 1293--1327 (2014) doi: 10.1007/s11118-014-9418-5

\bibitem{Monge1781}
Monge, G.: Mémoire sur la théorie des déblais et des remblais. Histoire de l'Académie Royale des Sciences de Paris, 666–-704 (1781)

\bibitem{monsaingeon2021new}
Monsaingeon, L.: A new transportation distance with bulk/interface interactions and flux penalization. Calculus of Variations and Partial Differential Equations, \textbf{60}(3) (2021) doi:10.1007/s00526-021-01946-2

\bibitem{mugnolo2019actually}
Mugnolo, D.: What is actually a metric graph?. arXiv preprint arXiv:1912.07549 (2019) doi: 10.48550/arXiv.1912.07549

\bibitem{Otto2001}
Otto, F.: The Geometry of Dissipative Evolution Equations: The Porous Medium Equation. Communications in Partial Differential Equations, \textbf{26}(1–2), 101--174 (2001) doi: 10.1081/PDE-100002243

\bibitem{Panaretos2019}
Panaretos, V.M., Zemel, Y.: Statistical Aspects of Wasserstein Distances. Annual Review of Statistics and Its Application, \textbf{6}, 405–-431 (2019) doi: 10.1146/annurev-statistics-030718-104938

\bibitem{PeyreCuturi2019}
Peyré, G., Cuturi, M.: Computational Optimal Transport: With Applications to Data Science. Foundations and Trends in Machine Learning, \textbf{11}(5–6), 355–-607 (2019) doi: 10.1561/2200000073

\bibitem{pooladian2024neural}
Pooladian, A.A., Domingo-Enrich, C., Chen, R.T.Q., Amos, B.: Neural optimal transport with Lagrangian costs. arXiv preprint arXiv:2406.00288 (2024) doi: 10.48550/arXiv.2406.00288

\bibitem{Santambrogio2015}
Santambrogio, F.: Optimal Transport for Applied Mathematicians. Birkhäuser (2015) doi: 10.1007/978-3-319-20828-2

\bibitem{villani2008optimal}
Villani, C.: Optimal Transport: Old and New. Springer, \textbf{338} (2008) doi: 10.1007/978-3-540-71050-9

\bibitem{weigand2025pdiffusion}
Weigand, L., Fazeny, A., Burger, M.: p-Diffustion equations on Graphs. In preparation (2025+)

\bibitem{zhang2023manifold}
Zhang, S., Mordant, G., Matsumoto, T., Schiebinger, G.: Manifold learning with sparse regularised optimal transport.arXiv preprint arXiv:2307.09816 (2023) doi: 10.48550/arXiv.2307.09816

\vspace{2cm}





%
%
%
%
%


\end{thebibliography}
